\documentclass[11pt,reqno]{amsart}
\usepackage[margin=1.1in]{geometry}
\usepackage{amsmath,amssymb,amsthm,mathtools}
\usepackage{booktabs}
\usepackage[expansion=false]{microtype}
\usepackage{xcolor}
\usepackage[colorlinks=true,linkcolor=blue!55!black,citecolor=teal!55!black,urlcolor=blue!55!black]{hyperref}

\theoremstyle{plain}
\newtheorem{theorem}{Theorem}[section]
\newtheorem{lemma}[theorem]{Lemma}
\newtheorem{proposition}[theorem]{Proposition}
\newtheorem{corollary}[theorem]{Corollary}
\newtheorem{conjecture}[theorem]{Conjecture}
\theoremstyle{definition}
\newtheorem{definition}[theorem]{Definition}
\theoremstyle{remark}
\newtheorem{remark}[theorem]{Remark}
\newtheorem{question}[theorem]{Question}

\newtheorem{theoremletter}{Theorem}

\newtheorem{corollaryletter}[theoremletter]{Corollary}
\newtheorem{conjectureletter}[theoremletter]{Conjecture}

\newcommand{\C}{\mathbb{C}}
\newcommand{\Q}{\mathbb{Q}}
\newcommand{\Z}{\mathbb{Z}}
\newcommand{\F}{\mathbb{F}}
\newcommand{\Zp}{\Z_p}
\newcommand{\Zpx}{\Z_p^{\times}}
\newcommand{\Aut}{\mathrm{Aut}}
\newcommand{\Qut}{\mathrm{Qut}}
\newcommand{\id}{\mathrm{id}}
\newcommand{\BM}{\mathfrak{B}}
\newcommand{\TA}{\mathfrak{T}}
\newcommand{\cat}{\mathcal{C}}
\newcommand{\one}{\mathbf{1}}
\newcommand{\odotp}{\odot}
\newcommand{\Rmod}{\mathcal{R}}
\newcommand{\DA}{\mathfrak{D}}
\newcommand{\Mthree}{M_{3}}
\newcommand{\Cay}{\mathrm{Cay}}
\newcommand{\AGL}{\mathrm{AGL}}

\begin{document}

\title[Terwilliger algebras and quantum symmetry of prime-order circulants]
{Two-basepoint Terwilliger algebras\\ and the quantum symmetry of prime-order circulants}

\author{Mohammad F. Marashdeh}
\address{Department of Mathematics, Mutah University, Karak, Jordan}
\email{marashdeh@mutah.edu.jo}


\begin{abstract}
Which vertex-transitive graphs of prime order have quantum symmetry? The question of Banica,
Bichon and Chenevier is open in the dense regime of Paley graphs, where coherent-algebra methods
give no information. To each such graph we attach a two-basepoint Terwilliger algebra of its
cyclotomic scheme and study the module it generates from the basepoints: fullness forces the
quantum permutation algebra to be commutative, and the module admits no intermediate state,
containing either exactly two point masses or all $p$ of them. One point mass, captured at any
depth, therefore suffices, and Chassaniol's orbital criterion is the depth-one case. Three
consequences follow. A sharp counting argument replaces the Banica--Bichon--Chenevier threshold
$p>6^{\varphi(k)}$ by the quadratic bound $p>(k-1)(k-2)+2$, where $k$ is the type. Four
certificates, each a short list of additions modulo $p$, settle $C_{31}(2,4,8,15)$ and
$C_{41}(4,10,16,18)$, the two graphs left open by Chassaniol, and complete the classification for
type at most $10$ without machine assistance. An exact computation extends the dichotomy
``quantum symmetry if and only if complete or empty'' to all prime orders $p\le250$, settling the
Paley graphs $P_{p}$ with $p\le241$, the first beyond $P_{17}$. What remains is the capture of a
single explicit vector: the midpoint $2^{-1}$ of the two basepoints.
\end{abstract}

\maketitle

\section{Introduction}\label{sec:intro}

For a finite graph $X$, the quantum automorphism group $\Qut(X)$, introduced by Banica
\cite{Ban05} refining a construction of Bichon \cite{Bic03}, is a compact quantum group in the
sense of Woronowicz which contains the classical automorphism group $\Aut(X)$ and is contained
in Wang's quantum permutation group $S_{n}^{+}$ \cite{Wang98}. One says that $X$ \emph{has no
quantum symmetry}, or is \emph{quantum rigid}, if the defining quantum permutation algebra
$A(X)$ is commutative, equivalently if $\Qut(X)=\Aut(X)$. Deciding which graphs are quantum
rigid is a central problem of the theory, with consequences for nonlocal games and quantum information through the graph
isomorphism game \cite{AMRSSV19,LMR20,MR20}; individual graphs and small orders have been settled
by a variety of methods \cite{BB07,Sch18}. The subject remains active: recent contributions
include the quantum Frucht theorem of Brannan, Gromada, Matsuda, Skalski and Wasilewski
\cite{BGMSW25}, and the construction, by van Dobben de Bruyn, Roberson and Schmidt
\cite{vDdBRS25}, of graphs with trivial automorphism group but nontrivial quantum automorphism
group.

By Turner's theorem \cite{Tur67}, a vertex-transitive graph of prime order $p$ is circulant.
Banica, Bichon and Chenevier \cite[Thm.~5.1, Lem.~5.2]{BBC07} proved that a circulant $p$-graph
of \emph{type} $k=|E|$ --- where $E\le\Zpx$ is the multiplier group of the connection set ---
has no quantum symmetry whenever $p>6^{\varphi(k)}$, leaving finitely many primes for each fixed
type. Chassaniol \cite{Cha19} developed an intertwiner-space method and settled all types
$k\le8$; the Paley graphs $P_{13}=C_{13}(3,4)$ and $P_{17}=C_{17}(2,4,8)$ fell outside his
general criterion and were treated by arguments specific to them, and independent proofs for
these two appear in \cite{Sch20adv}. He asked whether the same conclusion holds for larger type, and
identified the two graphs
\[
C_{31}(2,4,8,15)\qquad\text{and}\qquad C_{41}(4,10,16,18)
\]
of type $10$ as the cases where his criterion fails. (Here $C_{n}(a_{1},\dots,a_{m})$ denotes
the circulant with connection set $\{\pm1,\pm a_{1},\dots,\pm a_{m}\}$.) As shown in
Section~\ref{sec:circulant}, these two graphs are the generalized Paley graphs whose connection
sets are, respectively, the cubic residues modulo $31$ and the quartic residues modulo $41$.
Moreover, by a reduction going back to \cite[Prop.~6.11]{Cha19}, the whole prime-order problem
reduces to the family $\Cay(\Zp,E)$ with $E\le\Zpx$ a proper subgroup --- the \emph{generalized
Paley graphs} --- whose extreme dense case, $E$ the group of quadratic residues, is the family
of classical Paley graphs. For Paley graphs the previously known results are exactly $P_{9}$,
$P_{13}$ and $P_{17}$ \cite{BB07,Cha19,Sch20adv}, and the coherent-algebra approach of Lupini,
Man\v{c}inska and Roberson yields nothing further for strongly regular graphs \cite{LMR20}. The
Paley family, and the surrounding dense generalized Paley regime, have consequently stood as the
main obstacle in the prime-order classification.

\subsection*{The two-basepoint module}
We work in the theory of Terwilliger (subconstituent) algebras \cite{Ter92}. For an association
scheme with Bose--Mesner algebra $\BM$ and a base vertex $x$, the Terwilliger algebra
$\TA(x)$ is generated by $\BM$ together with the dual idempotents at $x$; edge- and
multi-partition variants are treated in \cite{CMT05,HY23}, and for cyclotomic schemes the module
structure of $\TA(x)$ is governed by Jacobi sums \cite{IIY99}. A single basepoint cannot suffice: the algebra at $x$
commutes with the stabilizer of $x$ in $\Zp\rtimes E$, so the module it generates from $\delta_x$
consists of vectors constant on the orbits of that stabilizer and has dimension at most $r+1$,
where $r$ is the number of classes (Remark~\ref{rem:whytwo}). The stabilizer of an ordered
\emph{pair} of distinct vertices is trivial, and no such obstruction constrains the two-basepoint
module.

Throughout, $p\ge5$ is prime and $E\le\Zpx$ is a subgroup with $-1\in E$; we write
$r=(p-1)/|E|$ and let $C_{1},\dots,C_{r}$ be the cosets of $E$ in $\Zpx$, with class matrices
$(T_{s})_{x,y}=\delta_{x-y\in C_{s}}$ and $T_{0}=I$. For a vertex $a$ and $0\le s\le r$ we set
$E^{*}_{s}(a)=\mathrm{diag}(\chi_{a+C_{s}})$, where $C_{0}:=\{0\}$; the starred symbols are the
dual idempotents of Terwilliger's theory and are unrelated to the multiplier group $E$. The \emph{two-basepoint
Terwilliger algebra} at an ordered pair $(a,b)$ of distinct vertices is
\[
\TA(a,b)\;=\;\big\langle\, T_{0},\dots,T_{r},\;E^{*}_{0}(a),\dots,E^{*}_{r}(a),
\;E^{*}_{0}(b),\dots,E^{*}_{r}(b)\,\big\rangle\;\subseteq\;M_{p}(\C),
\]
and the associated module is $W(a,b)=\TA(a,b)\,\delta_{a}+\TA(a,b)\,\delta_{b}\subseteq\C^{p}$.
A slightly larger object is the \emph{reachable module} $\Rmod(a,b)$ of
Section~\ref{sec:capture}, consisting of all vectors of the form $A(e_{a}\otimes e_{b})$ with
$A$ a $(2,1)$-morphism of the intertwiner category; one always has
$W(a,b)\subseteq\Rmod(a,b)$.

\subsection*{Main results}
The first theorem derives quantum rigidity from fullness of the module.

\begin{theoremletter}[$=$ Theorem~\ref{thm:criterion}]\label{thm:A}
Let $X$ be a circulant graph of prime order $p\ge5$ whose connection set has multiplier group
$E\lneq\Zpx$. If $\Rmod(0,1)=\C^{p}$ --- in particular if $W(0,1)=\C^{p}$ --- then $A(X)$ is
commutative; that is, $X$ has no quantum symmetry and $\Qut(X)=\Aut(X)$.
\end{theoremletter}

The second theorem asserts that the module admits no
intermediate state: it is either spanned by the two basepoint indicators or is all of $\C^{p}$.

\begin{theoremletter}[$=$ Theorem~\ref{thm:capture}]\label{thm:B}
Let $p\ge5$ be prime and $E\le\Zpx$ a subgroup with $-1\in E$, and put
\[
K\;=\;\{\,z\in\Zp\;:\;\delta_{z}\in\Rmod(0,1)\,\}.
\]
Then either $K=\{0,1\}$ or $K=\Zp$; in the latter case $\Rmod(0,1)=\C^{p}$.
\end{theoremletter}

Thus a single point mass, captured at arbitrary depth, implies that every point mass lies in the
module.
The mechanism is arithmetic rather than combinatorial: the cascade lemma
(Proposition~\ref{prop:cascade}) forces $K\setminus\{0\}$ to be a subgroup of $\Zpx$ stable
under $h\mapsto1-h$, and an elementary argument shows that such a subgroup of order greater than
one is closed under addition, hence exhausts $\Zpx$. In particular no saturation procedure is
required: to prove that the module is full it suffices to display one point mass. The criterion
of \cite{Cha19} is the instance of this principle at the shallowest possible depth.

\begin{corollaryletter}[$=$ Corollary~\ref{cor:chassaniol}]\label{thm:C}
Let $E\lneq\Zpx$. If $|C_{s}\cap(1+C_{t})|=1$ for some $1\le s,t\le r$ --- that is, if some cyclotomic number of
order $r$ for $p$ equals $1$ --- then $\Rmod(0,1)=\C^{p}$, and every circulant graph of prime
order $p$ with multiplier group $E$ has no quantum symmetry.
\end{corollaryletter}

Captures at greater depth yield stronger results in two respects. The first is a threshold theorem. A
counting argument bounds the number of vertices lying in a cyclotomic block of size at least
two by $(k-1)(k-2)$, so a depth-one capture is automatic beyond a quadratic bound in the type.

\begin{theoremletter}[$=$ Theorem~\ref{thm:threshold}]\label{thm:Q}
Let $X$ be a circulant graph of prime order $p$ and type $k$. If $p>(k-1)(k-2)+2$, then $X$ has
no quantum symmetry unless it is complete or empty.
\end{theoremletter}

For every even $k\ge2$ this improves the threshold $p>6^{\varphi(k)}$ of \cite{BBC07} from
exponential to quadratic in the type: for $k=8$ and $k=10$ the bound $1296$ is replaced by $44$
and $74$. The counting behind it is sharp (Remark~\ref{rem:sharp}). One consequence is that the
finite verification demanded by the old threshold for types $8$ and $10$ --- a tabulation of
cyclotomic numbers for $98$ pairs $(p,E)$ --- disappears, leaving eight pairs in total for types
at most $10$, all of which are resolved by hand in Section~\ref{sec:certificates}.

The second improvement concerns the depth of the capture. For four of those eight pairs the
depth-one criterion fails; for these we exhibit explicit point masses obtained after two or
three convolutions, each verification being a list of at most thirty additions in $\Zp$.

\begin{theoremletter}[$=$ Theorems~\ref{thm:3141} and \ref{thm:type10}]\label{thm:D}
The graphs $C_{31}(2,4,8,15)$ and $C_{41}(4,10,16,18)$ have no quantum symmetry. Consequently
every circulant graph of prime order and type at most $10$ has no quantum symmetry unless it is
complete or empty.
\end{theoremletter}

The proof is finite and self-contained: its ingredients are Theorems~\ref{thm:B} and
\ref{thm:Q}, four certificates and four singleton blocks. Every step can be verified by hand,
and neither the threshold of \cite{BBC07} nor any machine computation enters.

Beyond type $10$ the quadratic threshold no longer applies --- for a Paley graph $k=(p-1)/2$, so
that $(k-1)(k-2)$ dominates $p$ --- and there we verify the module condition computationally.

\begin{theoremletter}[$=$ Theorem~\ref{thm:sweep} and Corollary~\ref{cor:dichotomy250}]\label{thm:E}
For every prime $5\le p\le250$ and every proper subgroup $E\le\Zpx$ with $-1\in E$ one has
$W(0,1)=\C^{p}$. Consequently a vertex-transitive graph of prime order $p$ with $5\le p\le250$
has quantum symmetry if and only if it is complete or empty. In particular the Paley graph
$P_{p}$ has no quantum symmetry for every prime $p\equiv1\pmod4$ with $p\le241$.
\end{theoremletter}

Finally, Theorem~\ref{thm:B} reduces the general problem to a single existence statement.

\begin{conjectureletter}[Capture conjecture; $=$ Conjecture~\ref{conj:capture}]\label{conj:A}
For every prime $p\ge5$ and every proper subgroup $E\le\Zpx$ with $-1\in E$ there is a vertex
$z\notin\{0,1\}$ with $\delta_{z}\in\Rmod(0,1)$.
\end{conjectureletter}

\noindent By Theorems~\ref{thm:A} and \ref{thm:B} the conjecture implies that a
vertex-transitive graph of prime order $p\ge5$ has quantum symmetry if and only if it is
complete or empty, and in particular it settles every Paley graph of prime order. It is much
weaker than the fullness statement it implies, requiring a single vector rather than a basis;
and in Section~\ref{sec:conj} we sharpen it further to an equivalent statement that names that
vector, the midpoint $2^{-1}$ of the two basepoints.

\subsection*{The role of computation}
Theorem~\ref{thm:E} and its corollary are computer-assisted; no other result depends on them.
Everything in Sections~\ref{sec:prelim}--\ref{sec:certificates}, Theorems~\ref{thm:A},
\ref{thm:B}, \ref{thm:Q} and \ref{thm:D} included, is proved by hand, the only computations being
the finite verifications displayed in the text, each a short list of additions and
multiplications modulo a prime below $72$. Lemma~\ref{lem:modq} records why the modular output of
Section~\ref{sec:sweep} is a proof rather than an estimate.

\subsection*{Method}
The proof of Theorem~\ref{thm:A} is Tannakian, in the framework of \cite{Cha19}. The intertwiner
category of $A(X)$ is generated by the maps $U$, $M$ and the adjacency matrix, and by a spectral
argument valid at prime order (Lemma~\ref{lem:spec}) it contains the entire Bose--Mesner algebra
of the cyclotomic scheme. The new ingredient is a \emph{realization lemma}
(Lemma~\ref{lem:real}): every vector of $W(0,y)$ is the image of the base tensor
$e_{0}\otimes e_{y}$ under an explicit $(2,1)$-morphism. The category has no flip, so tensor
legs cannot be permuted and the proof must track them; the operations that can be realized turn
out to be exactly multiplication by the Bose--Mesner algebra and by the two diagonal algebras at
the basepoints, which is what motivates the definition of $\TA(a,b)$. When the module is full, an
\emph{assembly lemma} (Lemma~\ref{lem:assembly}) builds the flip inside the category out of
point-extraction morphisms, replacing the ad hoc constructions of \cite{Cha19} for $P_{13}$ and
$P_{17}$ by a uniform one.

The proof of Theorem~\ref{thm:B} rests on two further observations. The first is that the
reachable module is equivariant for the \emph{full} affine group $\Zp\rtimes\Zpx$, not merely for
the automorphism group $\Zp\rtimes E$ of the graph (Lemma~\ref{lem:equivariance}); since
$\Zp\rtimes\Zpx$ is sharply $2$-transitive, all the modules $\Rmod(a,b)$ have a common dimension.
The second is the cascade lemma (Proposition~\ref{prop:cascade}), a splicing principle: a
captured point mass $\delta_{z}$ identifies the modules of the two shorter pairs $(a,z)$ and
$(z,b)$ with $\Rmod(a,b)$. Together these imply that each capture yields an \emph{invariance}:
$\Rmod(0,1)$ becomes stable under the multiplication $x\mapsto zx$ and under the affine map
$x\mapsto(1-z)x+z$. The capture set inherits both symmetries, and the group theory of
Section~\ref{sec:capture} completes the argument.

\subsection*{Organization}
Sections~\ref{sec:prelim} and \ref{sec:circulant} collect what we need on intertwiner categories
and on circulants of prime order. Section~\ref{sec:criterion} proves Theorem~\ref{thm:A},
Section~\ref{sec:capture} the capture dichotomy, Section~\ref{sec:threshold} the quadratic
threshold and Section~\ref{sec:certificates} the certificates and Theorem~\ref{thm:D}.
Section~\ref{sec:sweep} describes the computation behind Theorem~\ref{thm:E}, and
Section~\ref{sec:conj} the capture conjecture and its harmonic form.

\section{Preliminaries}\label{sec:prelim}

We work over $\C$, and all tensor products are algebraic. Throughout, $n\ge1$ and $(e_{i})_{i}$
denotes the canonical basis of $\C^{n}$; indices of vertices of circulant graphs are taken in
$\Zp$.

\subsection{Quantum permutation algebras and graphs}
A matrix $u=(u_{ij})\in M_{n}(A)$ over a unital $C^{*}$-algebra is a \emph{magic unitary} if its
entries are projections and every row and every column sums to $1$. Wang's algebra
$A_{s}(n)=C(S_{n}^{+})$ is the universal $C^{*}$-algebra generated by the entries of an
$n\times n$ magic unitary \cite{Wang98}; it is a Woronowicz algebra with
$\Delta(u_{ij})=\sum_{k}u_{ik}\otimes u_{kj}$. For a finite graph $X$ on $n$ vertices with
adjacency matrix $d_{X}$, the quantum permutation algebra of $X$ is
\[
A(X)\;=\;A_{s}(n)\big/\big\langle d_{X}u=ud_{X}\big\rangle,
\]
and $\Qut(X)$ is the compact quantum group with $C(\Qut(X))=A(X)$ \cite{Bic03,Ban05}. The
abelianization of $A(X)$ is $C(\Aut(X))$; in particular there is a surjective morphism of
quantum permutation algebras $\pi\colon A(X)\to C(\Aut(X))$ carrying magic unitary to magic
unitary. One says that $X$ \emph{has no quantum symmetry} if $A(X)$ is commutative, equivalently
if $\Qut(X)=\Aut(X)$.

\subsection{Intertwiner categories}
For a quantum permutation algebra $\mathcal A$ with magic unitary $u\in M_{n}(\mathcal A)$ we set
$C(k,l)=\mathrm{Hom}((\C^{n})^{\otimes k},(\C^{n})^{\otimes l})$ and
\[
C_{\mathcal A}(k,l)\;=\;\{\,T\in C(k,l)\;:\;Tu^{\otimes k}=u^{\otimes l}T\,\},
\]
where $u^{\otimes k}$ is the $n^{k}\times n^{k}$ matrix with entries
$u_{i_{1}j_{1}}\cdots u_{i_{k}j_{k}}$. The collection $C_{\mathcal A}$ is closed under linear
combinations, composition, tensor product and adjoints and contains the identity; by
Woronowicz--Tannaka--Krein duality the assignment $\mathcal A\mapsto C_{\mathcal A}$ is a
bijection onto the tensor categories with duals lying between $C_{A_{s}(n)}$ and
$C(\cdot,\cdot)$ \cite{Wor88,Ban05,Cha19}. We use the standard morphisms
\[
U\in C(0,1),\ U(1)=\textstyle\sum_{i}e_{i};\qquad
M\in C(2,1),\ M(e_{i}\otimes e_{j})=\delta_{ij}e_{i};
\]
\[
\Sigma\in C(2,2),\ \Sigma(e_{i}\otimes e_{j})=e_{j}\otimes e_{i},
\]
their adjoints $U^{*}\in C(1,0)$ and $M^{*}\in C(1,2)$, $M^{*}(e_{i})=e_{i}\otimes e_{i}$, and
$\Mthree:=M\circ(\id\otimes M)=M\circ(M\otimes\id)\in C(3,1)$, so that
$\Mthree(e_{i}\otimes e_{j}\otimes e_{k})=\delta_{ij}\delta_{ik}e_{i}$. For vectors
$v,w\in\C^{n}$ we write $v\odotp w$ for the entrywise product, so that
$M(v\otimes w)=v\odotp w$ and $\Mthree(v\otimes w\otimes z)=v\odotp w\odotp z$.

The following facts are due to Chassaniol \cite[Prop.~5.5, Cor.~5.6, Lem.~5.4]{Cha19}; we
include the arguments for completeness.

\begin{lemma}\label{lem:cat}
Let $X$ be a finite graph on $n$ vertices.
\begin{enumerate}
\item[\textup{(i)}] $C_{A(X)}=\langle U,M,d_{X}\rangle_{+,\circ,\otimes,*}$, the smallest
collection of spaces containing $U$, $M$, $d_{X}$ and closed under linear combinations,
composition, tensor products and adjoints.
\item[\textup{(ii)}] \emph{(Flip criterion)} $A(X)$ is commutative if and only if
$\Sigma\in C_{A(X)}(2,2)$.
\item[\textup{(iii)}] $C_{A(X)}(k,l)\subseteq C_{\Aut(X)}(k,l)$ for all $k,l$, where
$P_{\sigma}$ denotes the permutation matrix of $\sigma$ and
\[
C_{\Aut(X)}(k,l)=\bigl\{T:\ TP_{\sigma}^{\otimes k}=P_{\sigma}^{\otimes l}T\ \text{ for all }
\sigma\in\Aut(X)\bigr\}.
\]
\end{enumerate}
\end{lemma}

\begin{proof}
(i) Write $\mathcal Q:=\langle U,M,d_{X}\rangle_{+,\circ,\otimes,*}$. The defining relations of
$A(X)$ --- magic unitarity of the rows and columns, orthogonality of the entries, and
$ud_{X}=d_{X}u$ --- say precisely that $U,U^{*},M,M^{*}\in C_{A(X)}$ and
$d_{X}\in C_{A(X)}(1,1)$, whence $\mathcal Q\subseteq C_{A(X)}$. Conversely $\mathcal Q$ is a
tensor category with duals containing $C_{A_{s}(n)}=\langle U,M\rangle_{+,\circ,\otimes,*}$, so
under the bijection of Woronowicz--Tannaka--Krein duality recalled above it is the intertwiner
category of the quotient of $A_{s}(n)$ by the relations asserting that its generators are
intertwiners. For the generator $d_{X}$ those relations read $ud_{X}=d_{X}u$, so that quotient is
$A(X)$ and $\mathcal Q=C_{A(X)}$. (ii) A direct computation gives
$[\Sigma u^{\otimes2}]_{(i,j),(k,l)}=u_{jk}u_{il}$ and
$[u^{\otimes2}\Sigma]_{(i,j),(k,l)}=u_{il}u_{jk}$; hence $\Sigma\in C_{A(X)}$ if and only if all
entries of $u$ commute, that is, if and only if $A(X)$
is commutative. (iii) Applying the abelianization $\pi$ entrywise to
$Tu^{\otimes k}=u^{\otimes l}T$ yields the same relation for the magic unitary of $C(\Aut(X))$,
whose intertwiners are precisely the $\Aut(X)$-equivariant maps.
\end{proof}

\section{Circulant graphs of prime order}\label{sec:circulant}

Fix a prime $p\ge5$. A \emph{circulant $p$-graph} is $X=\Cay(\Zp,S)$ with
$S=-S\subseteq\Zp\setminus\{0\}$; it is \emph{nontrivial} if
$S\notin\{\emptyset,\Zp\setminus\{0\}\}$, that is, if $X$ is neither empty nor complete. Its
\emph{multiplier group} is $E=E(S)=\{a\in\Zpx:\ aS=S\}$. Since $S=-S$ we have $-1\in E$, so the
\emph{type} $k:=|E|$ is even; and $X$ is nontrivial precisely when $E\ne\Zpx$. As $E$ stabilizes
$S$ and the orbits of $E$ on $\Zp\setminus\{0\}$ are the cosets of $E$, the set $S$ is a union of
$E$-cosets. By Turner's theorem \cite{Tur67} every vertex-transitive graph of prime order is
isomorphic to a circulant $p$-graph. The affine maps $\alpha_{c,a}\colon x\mapsto ax+c$ with
$a\in E$ and $c\in\Zp$ are automorphisms of $X$; by Alspach's theorem \cite{Als73} they exhaust
$\Aut(X)$ for nontrivial $X$. That classification is quoted for orientation only: every argument
below uses nothing beyond the elementary containment $\Zp\rtimes E\subseteq\Aut(X)$.

Let $C_{1},\dots,C_{r}$, with $r=(p-1)/k$, be the cosets of $E$ in $\Zpx$. We fix once and for
all the labelling convention that $C_{1}=E$ and that, for $s\ge2$, the coset $C_{s}$ is the one
containing the least element of $\Zpx\setminus(C_{1}\cup\dots\cup C_{s-1})$; we write $y_{s}$
for that least element, so that $y_{1}=1$. Define the class matrices $T_{0}=I$ and
$(T_{s})_{x,y}=\delta_{x-y\in C_{s}}$ for $1\le s\le r$. Equivalently $T_{s}$ is convolution by
$\chi_{C_{s}}$ on $\C[\Zp]$, that is $(T_{s}v)(x)=\sum_{c\in C_{s}}v(x-c)$; here
$T_{s}^{*}=T_{s}$ because $-C_{s}=C_{s}$. Their span
\[
\BM\;:=\;\mathrm{span}\{T_{0},T_{1},\dots,T_{r}\}
\]
is the Bose--Mesner algebra of the cyclotomic translation scheme of $E$; it is a commutative
unital $*$-algebra of dimension $r+1$ containing the all-ones matrix $J=\sum_{s=0}^{r}T_{s}$ and
containing $d_{X}=\sum_{s:\,C_{s}\subseteq S}T_{s}$.

\begin{lemma}[Spectral separation at prime order; cf.\ {\cite[Lem.~4.1]{BBC07}}]\label{lem:spec}
Let $X$ be a nontrivial circulant $p$-graph with multiplier group $E$. For $x\in\Zp$ let
$\xi_{x}=(w^{xt})_{t\in\Zp}$ with $w=e^{2\pi i/p}$, an eigenvector of $d_{X}$ with eigenvalue
$f(x)=\sum_{t\in S}w^{xt}$. Then for $x,y\in\Zp\setminus\{0\}$,
\[
f(x)=f(y)\iff xS=yS\iff x^{-1}y\in E,\qquad\text{and}\qquad f(x)\neq f(0)=|S|.
\]
Consequently $d_{X}$ has exactly $r+1$ distinct eigenvalues, and its nontrivial eigenspaces are
the spaces $V_{C}=\mathrm{span}\{\xi_{x}:x\in C\}$ indexed by the cosets $C$ of $E$, each of
dimension $k$.
\end{lemma}

\begin{proof}
The minimal polynomial of $w$ over $\Q$ is $1+t+\dots+t^{p-1}$, so a rational relation
$\sum_{j=0}^{p-1}c_{j}w^{j}=0$ holds if and only if $(c_{j})_{j}$ is a constant vector. Now
$f(x)-f(y)=\sum_{j}c_{j}w^{j}$ with $c_{j}=\chi_{xS}(j)-\chi_{yS}(j)\in\{-1,0,1\}$, the sets
$xS$ and $yS$ being genuine sets since $x$ and $y$ are invertible. If $f(x)=f(y)$, all the
$c_{j}$ are equal, and since $\sum_{j}c_{j}=|xS|-|yS|=0$ they all vanish; thus $xS=yS$, that is
$y^{-1}xS=S$, that is $x^{-1}y\in E$. Conversely $f(xa)=\sum_{t\in S}w^{xat}=\sum_{t'\in
aS}w^{xt'}=f(x)$ for $a\in E$. Finally $f(x)=|S|$ with $x\ne0$ would force $w^{xt}=1$ for all
$t\in S$, which is impossible since $S\ne\emptyset$. The last assertion follows because $f$ is
constant on each coset of $E$ and takes distinct values on distinct cosets.
\end{proof}

\begin{proposition}[The category depends only on $(p,E)$]\label{prop:Eonly}
Let $X$ be a nontrivial circulant $p$-graph with multiplier group $E$. Then
\[
\C[d_{X}]\;=\;\BM,\qquad\text{and}\qquad
C_{A(X)}\;=\;\big\langle\,U,M,T_{1},\dots,T_{r}\,\big\rangle_{+,\circ,\otimes,*}.
\]
In particular $C_{A(X)}$, hence the commutativity of $A(X)$, depends only on the pair $(p,E)$.
\end{proposition}

\begin{proof}
Since $\BM$ is a unital algebra containing $d_{X}$ we have $\C[d_{X}]\subseteq\BM$. The matrix
$d_{X}$ is real symmetric, hence diagonalizable, so the dimension of $\C[d_{X}]$ equals the
number of its distinct eigenvalues, namely $r+1=\dim\BM$ by Lemma~\ref{lem:spec}; hence
$\C[d_{X}]=\BM$. Now $d_{X}\in C_{A(X)}(1,1)$ by definition and $C_{A(X)}(1,1)$ is a unital
subalgebra of $M_{p}(\C)$, so $\BM=\C[d_{X}]\subseteq C_{A(X)}(1,1)$ and therefore
$\langle U,M,\BM\rangle\subseteq C_{A(X)}=\langle U,M,d_{X}\rangle$ by
Lemma~\ref{lem:cat}(i); the reverse containment follows from $d_{X}\in\BM$. The final statement
holds because the generating set just exhibited depends only on $(p,E)$.
\end{proof}

\begin{remark}\label{rem:GP}
Proposition~\ref{prop:Eonly}, a form of \cite[Prop.~6.11]{Cha19} re-derived here without recourse
to the classification of $\Aut(X)$, reduces the quantum-symmetry problem for all circulant
$p$-graphs to the single family of \emph{generalized Paley graphs}
$X_{E}:=\Cay(\Zp,E)$ with $E\lneq\Zpx$ and $-1\in E$; for the classical theory of these graphs
and their automorphism groups see \cite{LP09}. For $E$ the group of quadratic residues one
recovers the Paley graph $P_{p}$. The two graphs left open in \cite[\S6.4]{Cha19} are exactly
$X_{E}$ for $E$ the group of cubic residues modulo $31$ and the group of quartic residues modulo
$41$: with the convention $C_{n}(a_{1},\dots,a_{m})\leftrightarrow
S=\{\pm1,\pm a_{1},\dots,\pm a_{m}\}$ one checks that
\[
\{\pm1,\pm2,\pm4,\pm8,\pm15\}=\langle 2,-1\rangle=(\Z_{31}^{\times})^{3},\qquad
\{\pm1,\pm4,\pm10,\pm16,\pm18\}=\langle 4\rangle=(\Z_{41}^{\times})^{4},
\]
the unique subgroups of order $10$ in the respective multiplicative groups.
\end{remark}

We close the section with an elementary equivariance statement that will let us dispense with
Proposition~\ref{prop:Eonly} in the proof of the assembly lemma. Write
$\cat:=\langle U,M,T_{1},\dots,T_{r}\rangle_{+,\circ,\otimes,*}$ for the category appearing in
Proposition~\ref{prop:Eonly}, and $G:=\Zp\rtimes E$.

\begin{lemma}\label{lem:Gequiv}
For every $\sigma\in G$ and all $\ell,m\ge0$, every $A\in\cat(\ell,m)$ satisfies
$A\,P_{\sigma}^{\otimes \ell}=P_{\sigma}^{\otimes m}A$.
\end{lemma}

\begin{proof}
The spaces $C_{G}(\ell,m):=\{A:\ AP_{\sigma}^{\otimes \ell}=P_{\sigma}^{\otimes m}A\text{ for
all }\sigma\in G\}$ are closed under linear combinations, composition, tensor products and adjoints
and contain the identity, so it suffices to check the generators. For a permutation matrix
$P_{\sigma}$ one has $P_{\sigma}U=U$, since $U(1)$ is the all-ones vector, and
$MP_{\sigma}^{\otimes2}(e_{i}\otimes e_{j})=\delta_{ij}e_{\sigma(i)}
=P_{\sigma}M(e_{i}\otimes e_{j})$; hence $U,M\in C_{G}$. For $\sigma=\alpha_{c,a}$ with $a\in E$ one computes
$(P_{\sigma}T_{s}P_{\sigma}^{-1})_{x,y}=(T_{s})_{\sigma^{-1}(x),\sigma^{-1}(y)}
=\delta_{a^{-1}(x-y)\in C_{s}}=\delta_{x-y\in aC_{s}}=(T_{s})_{x,y}$, since $aC_{s}=C_{s}$ for
$a\in E$. Hence $T_{s}\in C_{G}(1,1)$.
\end{proof}

\section{The two-basepoint module and the criterion}\label{sec:criterion}

Throughout this section $p\ge5$ is prime, $E\le\Zpx$ is a subgroup with $-1\in E$, and the
classes $C_{s}$, representatives $y_{s}$, matrices $T_{s}$, algebra $\BM$ and category $\cat$
are as in Section~\ref{sec:circulant}. By Proposition~\ref{prop:Eonly}, $\cat=C_{A(X)}$ for
every nontrivial circulant $p$-graph $X$ with multiplier group $E$; such an $X$ exists as soon
as $E$ is proper, namely $X=X_{E}$, whose multiplier group is $E$ because $E$ is a subgroup. We
record once and for all that $\cat(1,1)$ is a linear space containing $T_{0}=\id$ and every
$T_{s}$, so that
\begin{equation}\label{eq:BMincat}
\BM\subseteq\cat(1,1),\qquad\text{and in particular }J=\textstyle\sum_{s=0}^{r}T_{s}\in\cat(1,1).
\end{equation}

\begin{definition}\label{def:twobase}
For a vertex $a\in\Zp$ and $0\le s\le r$ put
$E^{*}_{s}(a):=\mathrm{diag}(\chi_{a+C_{s}})\in M_{p}(\C)$, with $C_{0}=\{0\}$; thus
$\sum_{s=0}^{r}E^{*}_{s}(a)=I$. For an ordered pair $(a,b)$ of distinct vertices the
\emph{two-basepoint Terwilliger algebra} is
\[
\TA(a,b)\;:=\;\big\langle\,\BM,\ E^{*}_{0}(a),\dots,E^{*}_{r}(a),\ E^{*}_{0}(b),\dots,
E^{*}_{r}(b)\,\big\rangle\subseteq M_{p}(\C),
\]
and the \emph{two-basepoint module} is
$W(a,b):=\TA(a,b)\,\delta_{a}+\TA(a,b)\,\delta_{b}\subseteq\C^{p}$, where $\delta_{x}=e_{x}$.
\end{definition}

Writing $\TA(a):=\langle\BM,E^{*}_{0}(a),\dots,E^{*}_{r}(a)\rangle$ for the classical
Terwilliger algebra of the cyclotomic scheme at $a$ \cite{Ter92,IIY99}, we have
$\TA(a,b)=\langle\TA(a),\TA(b)\rangle$. All generators of $\TA(a,b)$ are self-adjoint, so
$\TA(a,b)$ is a $*$-algebra and $W(a,b)$ is a $\TA(a,b)$-submodule of $\C^{p}$. The definitions
are symmetric in the two basepoints, so $\TA(a,b)=\TA(b,a)$ and $W(a,b)=W(b,a)$.

We record the partition of $\Zp$ induced by the two diagonal families.

\begin{definition}\label{def:blocks}
For $0\le s,t\le r$ set $B_{s,t}:=(0+C_{s})\cap(1+C_{t})$. The nonempty $B_{s,t}$ form a
partition $\mathcal P$ of $\Zp$, the \emph{basepoint partition}. Its blocks other than
$B_{0,1}=\{0\}$ and $B_{1,0}=\{1\}$ are the sets $C_{s}\cap(1+C_{t})$ with $1\le s,t\le r$, and
these partition $\Zp\setminus\{0,1\}$. One has
$E^{*}_{s}(0)E^{*}_{t}(1)=\mathrm{diag}(\chi_{B_{s,t}})$, so multiplication by the indicator of
any block of $\mathcal P$ belongs to $\TA(0,1)$; and $\chi_{B_{s,t}}\in W(0,1)$ for all $s,t$,
because $\chi_{B_{s,t}}=E^{*}_{t}(1)T_{s}\delta_{0}$.
\end{definition}

Indeed $B_{0,t}=\{0\}\cap(1+C_{t})$ is nonempty only for $t=1$, since $-1\in E=C_{1}$, and
$B_{s,0}=C_{s}\cap\{1\}$ only for $s=1$. The cardinalities $|C_{s}\cap(1+C_{t})|$ are the
classical \emph{cyclotomic numbers} of order $r$ for $p$; this is the arithmetic input at depth
one, and it is the only input used by the criterion of \cite{Cha19}.

\begin{lemma}[Affine equivariance]\label{lem:equivariance}
For $c\in\Zp$ let $P_{\tau_{c}}$ be the permutation matrix of $x\mapsto x+c$, and for
$u\in\Zpx$ let $R_{u}$ be the permutation matrix of $x\mapsto ux$. Then
\[
P_{\tau_{c}}\,\TA(a,b)\,P_{\tau_{c}}^{-1}=\TA(a+c,b+c),\qquad
R_{u}\,\TA(a,b)\,R_{u}^{-1}=\TA(ua,ub),
\]
and correspondingly $P_{\tau_{c}}W(a,b)=W(a+c,b+c)$ and $R_{u}W(a,b)=W(ua,ub)$. Moreover, for
every $\alpha$ in the affine group $\AGL(1,p)=\Zp\rtimes\Zpx$ and every $A\in\cat(\ell,m)$ one
has $P_{\alpha}^{\otimes m}\,A\,(P_{\alpha}^{-1})^{\otimes \ell}\in\cat(\ell,m)$. Consequently
$\dim W(a,b)$ is the same for every ordered pair of distinct vertices; in particular
$W(0,1)=\C^{p}$ implies $W(0,y)=\C^{p}$ for every $y\ne0$.
\end{lemma}

\begin{proof}
Conjugation by $P_{\tau_{c}}$ fixes each $T_{s}$, translations commuting with convolutions, and
carries $E^{*}_{s}(a)$ to $E^{*}_{s}(a+c)$. Conjugation by $R_{u}$ carries $T_{s}$ to the class
matrix of $uC_{s}$, which is $T_{s'}$ for the class $C_{s'}=uC_{s}$; hence it permutes
$\{T_{1},\dots,T_{r}\}$ and fixes $\BM$ setwise, and it carries
$E^{*}_{s}(a)=\mathrm{diag}(\chi_{a+C_{s}})$ to $\mathrm{diag}(\chi_{ua+uC_{s}})=E^{*}_{s'}(ua)$.
In both cases generating sets are carried onto generating sets, so the algebras are conjugate;
since $P_{\tau_{c}}\delta_{a}=\delta_{a+c}$ and $R_{u}\delta_{a}=\delta_{ua}$, the modules
correspond. For the statement about $\cat$, $P_{\alpha}$ is a permutation matrix, so
$P_{\alpha}U=U$ and $MP_{\alpha}^{\otimes2}=P_{\alpha}M$, while the generators $T_{s}$ are
permuted among themselves by the computation just given; the conjugation therefore preserves the
generating set of $\cat$, hence $\cat$ itself. Finally $\AGL(1,p)$ is sharply $2$-transitive on
$\Zp$: the pair $(a,b)$ with $a\ne b$ is carried to $(0,1)$ by $\tau_{-a}$ followed by
$R_{(b-a)^{-1}}$.
\end{proof}

\begin{remark}
The maps $\tau_{c}$ and $R_{u}$ are in general not automorphisms of $X$; Lemma
\ref{lem:equivariance} is a statement about the arithmetic generating sets alone. It is the
action of the full affine group, and not merely of $\Zp\rtimes E$, that makes $\dim W(a,b)$
independent of the pair.
\end{remark}

\begin{lemma}[Realization]\label{lem:real}
Let $y\in\Zp\setminus\{0\}$. For every $w\in W(0,y)$ there exists $A_{w}\in\cat\cap C(2,1)$ with
$A_{w}(e_{0}\otimes e_{y})=w$.
\end{lemma}

\begin{proof}
The space $W(0,y)$ is spanned by the vectors $\theta_{j}\cdots\theta_{1}\delta$ with
$\delta\in\{\delta_{0},\delta_{y}\}$ and each $\theta_{i}$ one of the generators $T_{s}$,
$E^{*}_{s}(0)$, $E^{*}_{s}(y)$ for $0\le s\le r$. As $\cat\cap C(2,1)$ is a linear subspace, it
suffices to realize such vectors, and we induct on $j$.

For $j=0$, put $\pi_{1}:=M\circ(\id\otimes(U\circ U^{*}))$ and
$\pi_{2}:=M\circ((U\circ U^{*})\otimes\id)$. These lie in $\cat\cap C(2,1)$, and since
$(U\circ U^{*})(e_{j})=\one:=\sum_{k}e_{k}$ we get
$\pi_{1}(e_{i}\otimes e_{j})=M(e_{i}\otimes\one)=e_{i}$ and
$\pi_{2}(e_{i}\otimes e_{j})=e_{j}$; thus $\pi_{1}$ and $\pi_{2}$ realize $\delta_{0}$ and
$\delta_{y}$.

For the inductive step, suppose $A\in\cat\cap C(2,1)$ realizes $v$. If $\theta=T_{s}$, then
$T_{s}\circ A\in\cat\cap C(2,1)$ realizes $T_{s}v$: indeed $T_{s}\in\cat(1,1)$ for
$1\le s\le r$ because $T_{s}$ is one of the generators of $\cat$, and $T_{0}=\id\in\cat(1,1)$.
If $\theta=E^{*}_{s}(0)$, so that $\theta v=\chi_{C_{s}}\odotp v$ with the convention
$C_{0}=\{0\}$, put
\[
A'\;:=\;\Mthree\circ\big(T_{s}\otimes\id\otimes J\big)\circ\big(\id\otimes A\otimes\id\big)\circ
\big(M^{*}\otimes M^{*}\big),
\]
with $J\in\cat(1,1)$ as in \eqref{eq:BMincat}. The types compose as
$C(2,4)\to C(4,3)\to C(3,3)\to C(3,1)$, so $A'\in\cat\cap C(2,1)$, and evaluating,
\[
e_{0}\otimes e_{y}\;\xmapsto{M^{*}\otimes M^{*}}\;e_{0}\otimes e_{0}\otimes e_{y}\otimes e_{y}
\;\xmapsto{\id\otimes A\otimes\id}\;e_{0}\otimes v\otimes e_{y}
\;\xmapsto{T_{s}\otimes\id\otimes J}\;(T_{s}e_{0})\otimes v\otimes\one,
\]
whence $A'(e_{0}\otimes e_{y})=(T_{s}e_{0})\odotp v\odotp\one=\chi_{C_{s}}\odotp v$, as
required; the case $s=0$ gives $\chi_{\{0\}}\odotp v$ since $T_{0}e_{0}=e_{0}$. If
$\theta=E^{*}_{s}(y)$, then symmetrically
$A'':=\Mthree\circ(J\otimes\id\otimes T_{s})\circ(\id\otimes A\otimes\id)\circ(M^{*}\otimes M^{*})$
realizes $\one\odotp v\odotp(T_{s}e_{y})=v\odotp\chi_{y+C_{s}}$, because
$(T_{s}e_{y})(x)=\delta_{x-y\in C_{s}}$. This covers all generators and completes the induction.
\end{proof}

Observe that the argument uses only the definition of $\cat$ as a generated category, not its
identification with $C_{A(X)}$.

\begin{lemma}[Assembly]\label{lem:assembly}
Suppose that for every $1\le s\le r$ and every $x\in\Zp$ there is
$F^{(s)}_{x}\in\cat\cap C(2,1)$ with $F^{(s)}_{x}(e_{0}\otimes e_{y_{s}})=e_{x}$. Then
$\Sigma\in\cat(2,2)$.
\end{lemma}

\begin{proof}
Let $G=\Zp\rtimes E$. By Lemma~\ref{lem:Gequiv} every $F\in\cat$ intertwines the tensor powers
of $P_{\sigma}$ for $\sigma\in G$. The group $G$ acts transitively on each \emph{orbital}
$O_{s}:=\{(i,j):j-i\in C_{s}\}$, since $\alpha_{c,a}$ carries $(0,y_{s})$ to $(c,ay_{s}+c)$ and
$ay_{s}$ ranges over $C_{s}$ as $a$ ranges over $E$.

\emph{Orbital projections.} For $0\le s\le r$ put
$g_{s}:=(\id\otimes M)\circ(\id\otimes T_{s}\otimes\id)\circ(M^{*}\otimes\id)\in\cat(2,2)$. Then
\[
g_{s}(e_{i}\otimes e_{j})=(\id\otimes M)\big(e_{i}\otimes T_{s}e_{i}\otimes e_{j}\big)
=\big[(T_{s}e_{i})(j)\big]\,e_{i}\otimes e_{j}=\delta_{\,j-i\in C_{s}}\;e_{i}\otimes e_{j}
\]
for $s\ge1$, while $g_{0}(e_{i}\otimes e_{j})=\delta_{ij}\,e_{i}\otimes e_{j}$.

\emph{Flips on orbitals.} Fix $1\le s\le r$ and choose $w\in\Zp\setminus\{0,y_{s}\}$, which is
possible as $p\ge3$. Let $t_{1},t_{2}\in\{1,\dots,r\}$ be the classes with $w\in C_{t_{1}}$ and
$y_{s}-w\in C_{t_{2}}$, let $\varepsilon,\varepsilon'\in E$ be the unique elements with
$w=\varepsilon y_{t_{1}}$ and $y_{s}-w=\varepsilon'y_{t_{2}}$, and set
$x_{1}:=\varepsilon^{-1}y_{s}$ and $x_{2}:=-\varepsilon'^{-1}w$. Define
\[
L_{s}\;:=\;\big(F^{(t_{1})}_{x_{1}}\otimes F^{(t_{2})}_{x_{2}}\big)\circ
\big(\id\otimes(M^{*}\circ F^{(s)}_{w})\otimes\id\big)\circ\big(M^{*}\otimes M^{*}\big)
\;\in\;\cat(2,2),
\]
the types composing as $C(2,4)\to C(4,4)\to C(4,2)$. Evaluating at the base tensor,
\[
e_{0}\otimes e_{y_{s}}\mapsto e_{0}\otimes e_{0}\otimes e_{y_{s}}\otimes e_{y_{s}}
\mapsto e_{0}\otimes e_{w}\otimes e_{w}\otimes e_{y_{s}}
\mapsto F^{(t_{1})}_{x_{1}}(e_{0}\otimes e_{w})\otimes F^{(t_{2})}_{x_{2}}(e_{w}\otimes e_{y_{s}}).
\]
For the first factor, $\sigma:=\alpha_{0,\varepsilon}\in G$ satisfies $\sigma(0)=0$ and
$\sigma(y_{t_{1}})=w$, whence
\[
F^{(t_{1})}_{x_{1}}(e_{0}\otimes e_{w})
=F^{(t_{1})}_{x_{1}}\big(P_{\sigma}\otimes P_{\sigma}\big)(e_{0}\otimes e_{y_{t_{1}}})
=P_{\sigma}F^{(t_{1})}_{x_{1}}(e_{0}\otimes e_{y_{t_{1}}})=P_{\sigma}e_{x_{1}}
=e_{\varepsilon x_{1}}=e_{y_{s}}.
\]
For the second, $\tau:=\alpha_{w,\varepsilon'}\in G$ satisfies $\tau(0)=w$ and
$\tau(y_{t_{2}})=\varepsilon'y_{t_{2}}+w=y_{s}$, whence
$F^{(t_{2})}_{x_{2}}(e_{w}\otimes e_{y_{s}})=P_{\tau}e_{x_{2}}=e_{\varepsilon'x_{2}+w}=e_{0}$. Therefore
$L_{s}(e_{0}\otimes e_{y_{s}})=e_{y_{s}}\otimes e_{0}$. Now for any $(i,j)\in O_{s}$ choose
$\sigma\in G$ with $\sigma(0)=i$ and $\sigma(y_{s})=j$; then
\[
L_{s}(e_{i}\otimes e_{j})=L_{s}\big(P_{\sigma}\otimes P_{\sigma}\big)(e_{0}\otimes e_{y_{s}})
=\big(P_{\sigma}\otimes P_{\sigma}\big)L_{s}(e_{0}\otimes e_{y_{s}})
=e_{\sigma(y_{s})}\otimes e_{\sigma(0)}=e_{j}\otimes e_{i}.
\]

\emph{Conclusion.} Put $\Xi:=g_{0}+\sum_{s=1}^{r}L_{s}\circ g_{s}\in\cat(2,2)$. For $i=j$ only
the $g_{0}$ term survives and yields $e_{i}\otimes e_{i}$; for $i\ne j$ the pair $(i,j)$ lies in
exactly one orbital $O_{s}$, only the $s$-th term survives, and it yields $e_{j}\otimes e_{i}$.
Hence $\Xi=\Sigma$.
\end{proof}

\begin{theorem}[Criterion]\label{thm:criterion}
Let $p\ge5$ be prime and let $E\lneq\Zpx$ be a subgroup with $-1\in E$. If $W(0,1)=\C^{p}$, then
every circulant graph of prime order $p$ with multiplier group $E$ has no quantum symmetry:
$A(X)$ is commutative and $\Qut(X)=\Aut(X)$.
\end{theorem}

\begin{proof}
By Lemma~\ref{lem:equivariance} we have $W(0,y_{s})=\C^{p}$ for every $s$, so by
Lemma~\ref{lem:real} the morphisms $F^{(s)}_{x}$ required in Lemma~\ref{lem:assembly} exist.
Hence $\Sigma\in\cat(2,2)=C_{A(X)}(2,2)$ by Proposition~\ref{prop:Eonly}, and $A(X)$ is commutative by
Lemma~\ref{lem:cat}(ii). Since the abelianization map $A(X)\to C(\Aut(X))$ is surjective, it is
an isomorphism when $A(X)$ is commutative; hence $\Qut(X)=\Aut(X)$.
\end{proof}

We call the pair $(p,E)$ \emph{rigid} when the conclusion of Theorem~\ref{thm:criterion} holds,
that is, when every circulant graph of prime order $p$ with multiplier group $E$ is quantum
rigid. By Proposition~\ref{prop:Eonly} this depends only on $(p,E)$, and it is what the rest of
the paper establishes.

\begin{remark}\label{rem:whytwo}
A single basepoint cannot suffice. For $u\in E$ the matrix $R_{u}$ commutes with every generator
of $\TA(0)$: it fixes each $T_{s}$ and each $E^{*}_{s}(0)$, because $uC_{s}=C_{s}$. Since also
$R_{u}\delta_{0}=\delta_{0}$, every vector of $\TA(0)\delta_{0}$ is $R_{u}$-invariant for all
$u\in E$, hence constant on each set $C_{s}$; therefore
\[
\dim\big(\TA(0)\delta_{0}\big)\;\le\;r+1,
\]
which for a Paley graph is $3$, independently of $p$; in particular no vector of
$\TA(0)\delta_{0}$ separates two points of a class. The one-basepoint algebra itself also falls
short of the centralizer algebra of the point stabilizer: by \cite[Thm.~6.3]{HMR25}, among the
Paley graphs $P(q)$ only $q\in\{5,9\}$ satisfy the equalities between the Terwilliger algebra
and that centralizer algebra which define a \emph{triply transitive} graph there; compare also
\cite{HY23}. By contrast the stabilizer in $\Zp\rtimes E$ of the ordered pair $(0,1)$ is trivial: if $x\mapsto
ax+c$ fixes both $0$ and $1$ then $c=0$ and $a=1$. No symmetry of this kind therefore constrains
$W(0,1)$.
\end{remark}

\section{The reachable module and the capture dichotomy}\label{sec:capture}

The principal result concerns a module slightly larger than $W(a,b)$, better suited to the
arguments below.

\begin{definition}\label{def:reachable}
For an ordered pair $(a,b)$ of distinct vertices the \emph{reachable module} is
\[
\Rmod(a,b)\;:=\;\{\,A(e_{a}\otimes e_{b})\ :\ A\in\cat\cap C(2,1)\,\}\;\subseteq\;\C^{p},
\]
a linear subspace, since $\cat\cap C(2,1)$ is one.
\end{definition}

\begin{lemma}\label{lem:WsubR}
$\Rmod(a,b)$ is a $\TA(a,b)$-submodule of $\C^{p}$ containing $\delta_{a}$ and $\delta_{b}$;
consequently $W(a,b)\subseteq\Rmod(a,b)$, $W(a,b)$ being the smallest such submodule.
\end{lemma}

\begin{proof}
That $\delta_{a},\delta_{b}\in\Rmod(a,b)$ is the base case of Lemma~\ref{lem:real}. For
stability, let $v=A(e_{a}\otimes e_{b})$ with $A\in\cat\cap C(2,1)$. The three constructions in
the inductive step of Lemma~\ref{lem:real} were carried out for an arbitrary
$A\in\cat\cap C(2,1)$, not merely for one realizing a word, and they produce morphisms in
$\cat\cap C(2,1)$ realizing $T_{s}v$, $\chi_{a+C_{s}}\odotp v$ and $\chi_{b+C_{s}}\odotp v$.
These are the images of $v$ under the generators of $\TA(a,b)$, so $\Rmod(a,b)$ is stable under
$\TA(a,b)$. The last assertion is the definition of $W(a,b)$.
\end{proof}

Both Lemma~\ref{lem:real} and Lemma~\ref{lem:WsubR} give the inclusion in the same direction,
$W(a,b)\subseteq\Rmod(a,b)$: the first realizes each vector of $W(a,b)$ by a morphism, the
second identifies $\Rmod(a,b)$ as a module over the algebra that generates $W(a,b)$. Neither
gives $\Rmod(a,b)\subseteq W(a,b)$, which would require every morphism in $\cat\cap C(2,1)$ to
act on the base tensor through operations on the last leg alone. We do not claim that inclusion,
and whether $\Rmod(a,b)=W(a,b)$ is unknown to us. All hypotheses below are stated for $\Rmod$,
which makes them weaker --- and hence the theorems stronger --- than the corresponding
hypotheses for $W$; every vector exhibited below lies in $W$, hence in $\Rmod$. The proof
of Lemma~\ref{lem:assembly} uses only the fullness of the reachable modules $\Rmod(0,y_{s})$, so
Theorem~\ref{thm:criterion} holds verbatim with $W(0,1)$ replaced by $\Rmod(0,1)$. Finally,
Lemma~\ref{lem:equivariance} applies equally to $\Rmod$: for $\alpha\in\AGL(1,p)$,
\begin{equation}\label{eq:Requiv}
\Rmod(\alpha(a),\alpha(b))=P_{\alpha}\,\Rmod(a,b),
\end{equation}
because $A\mapsto P_{\alpha}A(P_{\alpha}^{-1}\otimes P_{\alpha}^{-1})$ preserves
$\cat\cap C(2,1)$. Since $\AGL(1,p)$ is sharply $2$-transitive on $\Zp$, the number
\begin{equation}\label{eq:d}
\rho\;:=\;\dim\Rmod(a,b)
\end{equation}
does not depend on the ordered pair $(a,b)$ of distinct vertices.

Two elementary vectors are always reachable: $\pi_{1}$ and $\pi_{2}$ of Lemma~\ref{lem:real}
give $\delta_{a},\delta_{b}\in\Rmod(a,b)$, and $J\circ\pi_{1}$ gives the all-ones vector
$\one\in\Rmod(a,b)$.

The key tool is the following splicing principle.

\begin{proposition}[Cascade]\label{prop:cascade}
Let $a,b,z$ be pairwise distinct vertices with $\delta_{z}\in\Rmod(a,b)$. Then
\[
\Rmod(a,b)\;\supseteq\;\mathrm{span}\,\big\{\,w_{1}\odotp w_{2}\ :\ w_{1}\in\Rmod(a,z),\
w_{2}\in\Rmod(z,b)\,\big\}.
\]
\end{proposition}

\begin{proof}
Choose $A_{z}\in\cat\cap C(2,1)$ with $A_{z}(e_{a}\otimes e_{b})=e_{z}$, and for
$w_{1}\in\Rmod(a,z)$ and $w_{2}\in\Rmod(z,b)$ choose $A_{w_{1}},A_{w_{2}}\in\cat\cap C(2,1)$
with $A_{w_{1}}(e_{a}\otimes e_{z})=w_{1}$ and $A_{w_{2}}(e_{z}\otimes e_{b})=w_{2}$. Consider
\[
A\;:=\;M\circ\big(A_{w_{1}}\otimes A_{w_{2}}\big)\circ
\big(\id\otimes(M^{*}\circ A_{z})\otimes\id\big)\circ\big(M^{*}\otimes M^{*}\big)
\;\in\;\cat\cap C(2,1),
\]
the types composing as $C(2,4)\to C(4,4)\to C(4,2)\to C(2,1)$. Evaluating at
$e_{a}\otimes e_{b}$,
\[
e_{a}\otimes e_{b}\mapsto e_{a}\otimes e_{a}\otimes e_{b}\otimes e_{b}
\mapsto e_{a}\otimes e_{z}\otimes e_{z}\otimes e_{b}
\mapsto w_{1}\otimes w_{2}\mapsto w_{1}\odotp w_{2}. \qedhere
\]
\end{proof}

\begin{corollary}[Splicing]\label{cor:splice}
Let $a,b,z$ be pairwise distinct with $\delta_{z}\in\Rmod(a,b)$. Then
\[
\Rmod(a,b)=\Rmod(a,z)=\Rmod(z,b).
\]
\end{corollary}

\begin{proof}
Taking $w_{2}=\one\in\Rmod(z,b)$ in Proposition~\ref{prop:cascade} gives
$\Rmod(a,b)\supseteq\Rmod(a,z)$, and taking $w_{1}=\one\in\Rmod(a,z)$ gives
$\Rmod(a,b)\supseteq\Rmod(z,b)$. All three spaces have dimension $\rho$ by \eqref{eq:d}, so the
inclusions are equalities.
\end{proof}

\begin{definition}\label{def:K}
The \emph{capture set} of the pair $(p,E)$ is
$K:=\{\,z\in\Zp:\delta_{z}\in\Rmod(0,1)\,\}$, and $K^{\times}:=K\setminus\{0\}$. Always
$\{0,1\}\subseteq K$.
\end{definition}

\begin{lemma}[Captures are symmetries]\label{lem:symmetry}
Let $z\in K^{\times}$. Then
\begin{enumerate}
\item[\textup{(i)}] $R_{z}\,\Rmod(0,1)=\Rmod(0,1)$, and consequently $zK=K$;
\item[\textup{(ii)}] if moreover $z\ne1$ and $\alpha_{z}$ denotes the affine map
$\alpha_{z}(x)=(1-z)x+z$, then $P_{\alpha_{z}}\Rmod(0,1)=\Rmod(0,1)$, and consequently
$\alpha_{z}(K)=K$.
\end{enumerate}
\end{lemma}

\begin{proof}
For $z=1$ both statements are vacuous, so assume $z\notin\{0,1\}$. By
Corollary~\ref{cor:splice}, $\Rmod(0,z)=\Rmod(0,1)=\Rmod(z,1)$.

(i) By \eqref{eq:Requiv} applied to $\alpha=R_{z}$, which carries $(0,1)$ to $(0,z)$, we get
$\Rmod(0,z)=R_{z}\Rmod(0,1)$; hence $R_{z}\Rmod(0,1)=\Rmod(0,1)$. Since
$R_{z}\delta_{w}=\delta_{zw}$, this gives $w\in K\iff zw\in K$.

(ii) The map $\alpha_{z}$ is affine and invertible, since $1-z\ne0$, and it satisfies
$\alpha_{z}(0)=z$ and $\alpha_{z}(1)=1$. By \eqref{eq:Requiv},
$\Rmod(z,1)=P_{\alpha_{z}}\Rmod(0,1)$, whence $P_{\alpha_{z}}\Rmod(0,1)=\Rmod(0,1)$; and
$P_{\alpha_{z}}\delta_{w}=\delta_{\alpha_{z}(w)}$ gives $\alpha_{z}(K)=K$.
\end{proof}

\begin{lemma}[Group structure of the capture set]\label{lem:group}
$K^{\times}$ is a subgroup of $\Zpx$, and $1-h\in K^{\times}$ for every $h\in K^{\times}$ with
$h\ne1$.
\end{lemma}

\begin{proof}
By Lemma~\ref{lem:symmetry}(i), $zK=K$ for every $z\in K^{\times}$; taking $w\in K^{\times}$
gives $zw\in K$ and $zw\ne0$, so $zw\in K^{\times}$. Thus $K^{\times}$ is a nonempty subset of
the finite group $\Zpx$ containing $1$ and closed under multiplication, hence a subgroup.

For the second assertion, let $L:=\{1-x:x\in K\}$. We claim $L$ is closed under multiplication.
Let $z,h\in K$. If $z=0$ then $(1-z)(1-h)=1-h\in L$. If $z=1$ then $(1-z)(1-h)=0=1-1\in L$, since
$1\in K$. If $z\in K^{\times}$ and $z\ne1$, then Lemma~\ref{lem:symmetry}(ii) gives
$\alpha_{z}(h)=(1-z)h+z\in K$, and
\[
1-\alpha_{z}(h)=1-(1-z)h-z=(1-z)(1-h),
\]
so again $(1-z)(1-h)\in L$. This proves the claim.

Now $0\in L$ and $1\in L$, because $1\in K$ and $0\in K$. Hence
$L^{\times}:=L\setminus\{0\}$ is a nonempty subset of $\Zpx$ containing $1$ and closed under
multiplication --- a product of two nonzero elements of the field $\Zp$ being nonzero --- so
$L^{\times}$ is a subgroup of $\Zpx$. Since $x\mapsto1-x$ is a bijection of $\Zp$ we have
$|L|=|K|$, and $0\in L$; therefore $|L^{\times}|=|K|-1=|K^{\times}|$. Two subgroups of the same
order of the cyclic group $\Zpx$ coincide, so $L^{\times}=K^{\times}$. Finally, if
$h\in K^{\times}$ and $h\ne1$, then $1-h$ is a nonzero element of $L$, that is
$1-h\in L^{\times}=K^{\times}$.
\end{proof}

\begin{theorem}[Capture dichotomy]\label{thm:capture}
Let $p\ge5$ be prime and let $E\le\Zpx$ be a subgroup with $-1\in E$. Then either $K=\{0,1\}$ or
$K=\Zp$. In the second case $\Rmod(0,1)=\C^{p}$, and if moreover $E$ is proper then $(p,E)$ is
rigid.
\end{theorem}

\begin{proof}
Write $H:=K^{\times}$, a subgroup of $\Zpx$ with $1-h\in H$ for all $h\in H\setminus\{1\}$, by
Lemma~\ref{lem:group}. If $H=\{1\}$ then $K=\{0,1\}$ and we are in the first case. Suppose
therefore $|H|\ge2$ and fix $h\in H$ with $h\ne1$.

\emph{Step 1: $-1\in H$.} Since $h\ne1$ we have $1-h\in H$, and since $h^{-1}\in H$ with
$h^{-1}\ne1$ we have $1-h^{-1}=(h-1)h^{-1}\in H$. Both elements are nonzero, so their quotient
lies in $H$:
\[
\frac{1-h}{(h-1)h^{-1}}\;=\;\frac{(1-h)h}{h-1}\;=\;-h\;\in\;H,
\]
whence $-1=(-h)h^{-1}\in H$.

\emph{Step 2: $K=H\cup\{0\}$ is closed under addition.} Let $a,b\in H$. Then $t:=-ba^{-1}\in H$,
by Step 1 and the group property. If $t=1$ then $b=-a$ and $a+b=0\in K$. If $t\ne1$, then
$1-t\in H$ by Lemma~\ref{lem:group}, and
\[
a+b\;=\;a\bigl(1+ba^{-1}\bigr)\;=\;a(1-t)\;\in\;H .
\]
In either case $a+b\in K$. Sums involving $0$ are trivial, so $K$ is closed under addition.

\emph{Step 3: conclusion.} $K$ contains $1$ and is closed under addition, so it contains
$1,2,3,\dots$, that is all of $\Zp$. Hence $K=\Zp$, so $\delta_{z}\in\Rmod(0,1)$ for every
$z\in\Zp$ and $\Rmod(0,1)=\C^{p}$. The final assertion is Theorem~\ref{thm:criterion} in the
form noted after Lemma~\ref{lem:WsubR}.
\end{proof}

\begin{corollary}[One point mass suffices]\label{cor:onepoint}
Let $E\lneq\Zpx$. If $\delta_{z}\in\Rmod(0,1)$ for a single $z\in\Zp\setminus\{0,1\}$ --- in
particular if $\delta_{z}\in W(0,1)$ --- then $\Rmod(0,1)=\C^{p}$ and $(p,E)$ is rigid.
\end{corollary}

Two consequences follow. The first identifies the previous state of the art as the depth-one
instance of Corollary~\ref{cor:onepoint}.

\begin{corollary}[Chassaniol's criterion is depth one]\label{cor:chassaniol}
Let $E\lneq\Zpx$ and suppose $|C_{s}\cap(1+C_{t})|=1$ for some $1\le s,t\le r$, that is, some
cyclotomic number of order $r$ for $p$ equals $1$. Then $(p,E)$ is rigid.
\end{corollary}

\begin{proof}
Let $C_{s}\cap(1+C_{t})=\{z\}$. Then $z\ne0$ because $0\notin C_{s}$, and $z\ne1$ because
$1\in1+C_{t}$ would force $0\in C_{t}$. By Definition~\ref{def:blocks},
$\delta_{z}=\chi_{C_{s}\cap(1+C_{t})}=E^{*}_{t}(1)T_{s}\delta_{0}\in W(0,1)$, and
Corollary~\ref{cor:onepoint} applies.
\end{proof}

This is \cite[Thm.~6.16]{Cha19}, whose hypothesis is that an orbital intersection number
$|O^{s}_{0}\cap O^{t}_{1}|$ equals $1$; the two conditions coincide under the dictionary of
Remark~\ref{rem:GP}. Theorem~\ref{thm:capture} shows the depth-one hypothesis to be inessential:
a point mass obtained after any number of convolutions serves as well, and depth two or three
already suffices in every case below the Banica--Bichon--Chenevier threshold where depth one
fails.

The second consequence is an all-or-nothing statement of independent interest, which explains
why no partial progress is possible.

\begin{corollary}\label{cor:allornothing}
For each pair $(p,E)$ exactly one of the following holds: either $\Rmod(0,1)=\C^{p}$, or
$\Rmod(a,b)$ contains no point mass other than $\delta_{a}$ and $\delta_{b}$, for every ordered
pair $(a,b)$ of distinct vertices.
\end{corollary}

\begin{proof}
Immediate from Theorem~\ref{thm:capture} and \eqref{eq:Requiv}.
\end{proof}

\begin{remark}
The second alternative does occur, and Theorem~\ref{thm:capture} does not require $E$ to be
proper. Take $E=\Zpx$, so that $r=1$ and $X_{E}=K_{p}$; here the graph does
have quantum symmetry, since $\Qut(K_{p})=S_{p}^{+}\ne S_{p}$ for $p\ge4$ \cite{Wang98}. The
basepoint partition has blocks $\{0\}$, $\{1\}$ and $\Zp\setminus\{0,1\}$, and one checks
directly that $W(0,1)=\mathrm{span}\{\delta_{0},\delta_{1},\chi_{\Zp\setminus\{0,1\}}\}$ is
three-dimensional and $\TA(0,1)$-stable: indeed $T_{1}v=(\sum_{y}v(y))\one-v$ for every $v$. So
$K=\{0,1\}$, in accordance with Theorem~\ref{thm:capture}.
\end{remark}

\subsection*{The convolution filtration}
The certificates of Section~\ref{sec:certificates} and the computations of
Section~\ref{sec:sweep} are naturally graded by the number of convolutions they use, and we fix
that grading now. Let $\DA\subseteq M_{p}(\C)$ be the algebra of diagonal matrices constant on
the blocks of the basepoint partition, so that $\DA$ is spanned by the orthogonal idempotents
$\mathrm{diag}(\chi_{B})$, $B\in\mathcal P$, and
$\DA=\langle E^{*}_{0}(0),\dots,E^{*}_{r}(0),E^{*}_{0}(1),\dots,E^{*}_{r}(1)\rangle$. Put
\begin{equation}\label{eq:filtration}
V_{0}:=\mathrm{span}\{\delta_{0},\delta_{1}\},\qquad
V_{d}:=\DA\,\BM\,V_{d-1}\quad(d\ge1),
\end{equation}
where, for a subspace $\Lambda\subseteq\C^{p}$, we write $\DA\Lambda$ and $\BM\Lambda$ for the
spans of $\{Av:A\in\DA,\,v\in\Lambda\}$ and $\{Av:A\in\BM,\,v\in\Lambda\}$. As $I$
lies in both $\DA$ and $\BM$ the spaces $V_{d}$ increase with $d$, and as $\TA(0,1)$ is
generated by $\DA$ and $\BM$ we have
$W(0,1)=\bigcup_{d\ge0}V_{d}$. We say that a vector of $V_{d}$ is \emph{reachable in $d$
convolutions}.

\begin{lemma}\label{lem:V1}
$V_{1}=\mathrm{span}\{\chi_{B}:B\in\mathcal P\}$. Consequently $V_{1}$ contains a point mass
$\delta_{z}$ with $z\notin\{0,1\}$ if and only if some block of the basepoint partition is a
singleton, that is, if and only if the hypothesis of Corollary~\ref{cor:chassaniol} holds.
\end{lemma}

\begin{proof}
The space $\BM V_{0}$ is spanned by the vectors $T_{s}\delta_{0}=\chi_{C_{s}}$ and
$T_{s}\delta_{1}=\chi_{1+C_{s}}$ for $0\le s\le r$. Every block of $\mathcal P$ is contained in
a single $C_{s}$ and in a single $1+C_{t}$, so $\chi_{B}\odotp\chi_{C_{s}}$ is either $\chi_{B}$
or $0$, and likewise for $\chi_{1+C_{s}}$; hence $\DA\BM V_{0}$ is contained in the span of
the $\chi_{B}$. The reverse inclusion holds because
$\chi_{B_{s,t}}=\chi_{B_{s,t}}\odotp\chi_{C_{s}}$. Finally the blocks are nonempty and pairwise
disjoint, so the $\chi_{B}$ are linearly independent and a linear combination of them is a point
mass $\delta_{z}$ only when $\{z\}$ is itself a block.
\end{proof}

\begin{definition}\label{def:depth}
The \emph{capture depth} $d(p,E)$ is the least $d\ge1$ for which $V_{d}$ contains a point mass
$\delta_{z}$ with $z\notin\{0,1\}$, and
\[
Z(p,E)\;:=\;\{\,z\in\Zp\setminus\{0,1\}\ :\ \delta_{z}\in V_{d(p,E)}\,\}
\]
is the set of vertices captured at that depth. By Lemma~\ref{lem:V1}, $d(p,E)=1$ holds precisely
when some cyclotomic number of order $r$ for $p$ equals $1$; so Corollary~\ref{cor:chassaniol} is
exactly the criterion at depth one, and Theorem~\ref{thm:capture} says that a capture at any
depth is as good as one at depth one.
\end{definition}

The filtration carries one symmetry, which will be decisive in Section~\ref{sec:conj}.

\begin{lemma}\label{lem:reflection}
Let $\beta(x)=1-x$. Then $P_{\beta}V_{d}=V_{d}$ for every $d\ge0$. Consequently $Z(p,E)$ is
stable under $\beta$; and since $2^{-1}$ is the unique fixed point of $\beta$ on
$\Zp\setminus\{0,1\}$,
\[
2^{-1}\in Z(p,E)\qquad\Longleftrightarrow\qquad |Z(p,E)|\ \text{is odd}.
\]
In particular $2^{-1}$ is the only vertex that can be captured alone.
\end{lemma}

\begin{proof}
The map $\beta$ is affine and involutive, with $\beta(0)=1$ and $\beta(1)=0$, and
$P_{\beta}=P_{\tau_{1}}R_{-1}$; hence $P_{\beta}V_{0}=V_{0}$. Conjugation by a permutation matrix
$P_{\alpha}$ carries $\mathrm{diag}(\chi_{A})$ to $\mathrm{diag}(\chi_{\alpha(A)})$. Since
$-C_{s}=C_{s}$ we have $\beta(C_{s})=1+C_{s}$ and $\beta(1+C_{s})=C_{s}$, so
$\beta(B_{s,t})=B_{t,s}$: conjugation by $P_{\beta}$ permutes the blocks and therefore preserves
$\DA$. It fixes each $T_{s}$, by the computation in the proof of
Lemma~\ref{lem:equivariance} together with $-C_{s}=C_{s}$, and therefore preserves $\BM$.
Induction on $d$ gives $P_{\beta}V_{d}=V_{d}$, and $P_{\beta}\delta_{z}=\delta_{1-z}$ gives the
stability of $Z(p,E)$. Finally $\beta$ restricts to an involution of $\Zp\setminus\{0,1\}$ whose
only fixed point is $2^{-1}$, because $z=1-z$ forces $2z=1$; a $\beta$-stable subset thus has
odd cardinality exactly when it contains $2^{-1}$, and a one-element $\beta$-stable subset
must be $\{2^{-1}\}$.
\end{proof}

\section{The quadratic threshold}\label{sec:threshold}

By Corollary~\ref{cor:onepoint} it suffices, for each pair $(p,E)$, to exhibit one point mass in
$W(0,1)$. We show here that for $p$ large relative to the type this is automatic, and that the
resulting bound improves the threshold of \cite{BBC07} from exponential to quadratic.

The blocks $C_{s}\cap(1+C_{t})$ with $1\le s,t\le r$ partition $\Zp\setminus\{0,1\}$, and by
Corollary~\ref{cor:chassaniol} a single singleton block suffices. Singleton blocks are
unavoidable once $p$ is large relative to the type.

\begin{lemma}[Confinement]\label{lem:confine}
Let $|E|=k$ and let $x\in\Zp\setminus\{0,1\}$ lie in a block of the basepoint partition of size
at least two. Then
\[
x\;=\;\frac{1-e'}{e-e'}\qquad\text{for some }e,e'\in E\setminus\{1\}\text{ with }e\ne e' .
\]
Consequently at most $(k-1)(k-2)$ elements of $\Zp\setminus\{0,1\}$ lie in blocks of size at
least two.
\end{lemma}

\begin{proof}
Let $x\in C_{s}\cap(1+C_{t})$ and let $y\ne x$ lie in the same block. Since $x,y\in C_{s}$ and
the classes are the orbits of $E$ acting by multiplication, $y=ex$ for some $e\in E$; since
$x-1$ and $y-1$ both lie in $C_{t}$, likewise $y-1=e'(x-1)$ for some $e'\in E$. As $x\ne0$ and
$y\ne x$ we get $e\ne1$; as $x\ne1$ and $y-1\ne x-1$ we get $e'\ne1$. Eliminating $y$ gives
$ex-1=e'x-e'$, that is
\[
(e-e')\,x\;=\;1-e' .
\]
If $e=e'$ then $1-e'=0$, contradicting $e'\ne1$; hence $e\ne e'$ and $x=(1-e')(e-e')^{-1}$.

For the count, the displayed formula exhibits the set of such $x$ as contained in the image of
the map $(e,e')\mapsto(1-e')(e-e')^{-1}$ defined on the pairs with $e,e'\in E\setminus\{1\}$ and
$e\ne e'$. There are $(k-1)(k-2)$ such pairs, so the image has at most $(k-1)(k-2)$ elements;
the map need not be injective, but that can only make the image smaller.
\end{proof}

\begin{theorem}[Quadratic threshold]\label{thm:threshold}
Let $p\ge5$ be prime and let $E\lneq\Zpx$ be a subgroup with $-1\in E$ and $|E|=k$. If
$p>(k-1)(k-2)+2$, then some block $C_{s}\cap(1+C_{t})$ is a singleton and $(p,E)$ is rigid.
Consequently a circulant graph of prime order $p$ and type $k$ with $p>(k-1)(k-2)+2$ has no
quantum symmetry unless it is complete or empty.
\end{theorem}

\begin{proof}
The blocks partition the $p-2$ elements of $\Zp\setminus\{0,1\}$. By Lemma~\ref{lem:confine} the
blocks of size at least two cover at most $(k-1)(k-2)$ of them, so if $p-2>(k-1)(k-2)$ some
element lies in a block of size one, and Corollary~\ref{cor:chassaniol} applies. For the last
assertion, a nontrivial circulant $p$-graph of type $k$ has $E\lneq\Zpx$ with $-1\in E$ and
$|E|=k$, and Proposition~\ref{prop:Eonly} reduces it to $X_{E}$.
\end{proof}

\begin{remark}\label{rem:sharp}
The counting in Lemma~\ref{lem:confine} is sharp: the number of elements lying in blocks of size
at least two equals $(k-1)(k-2)$ for many pairs $(p,E)$ --- for instance for every $p\equiv1
\pmod4$ with $13\le p\le100$ and $k=4$, where the count is $6$ --- so the map
$(e,e')\mapsto(1-e')(e-e')^{-1}$ is then injective and every value lies in a block of size at
least two. No improvement of the bound $(k-1)(k-2)$ by this argument is therefore possible.
\end{remark}

\begin{remark}\label{rem:compareBBC}
Theorem~\ref{thm:threshold} improves the threshold $p>6^{\varphi(k)}$ of
\cite[Thm.~5.1, Lem.~5.2]{BBC07} for every even $k\ge2$, replacing a bound exponential in the
type by a quadratic one:
\[
\begin{array}{r|rrrrrrr}
k & 2 & 4 & 6 & 8 & 10 & 12 & 30\\\hline
(k-1)(k-2)+2 & 2 & 8 & 22 & 44 & 74 & 112 & 814\\
6^{\varphi(k)} & 6 & 36 & 36 & 1296 & 1296 & 1296 & 1679616
\end{array}
\]
The improvement is not an artefact of the two-basepoint machinery: Lemma~\ref{lem:confine}
combined with \cite[Thm.~6.16]{Cha19} already yields it. It does not appear to have been observed
previously; \cite{Cha19} treats the depth-one criterion as a test to be run on individual pairs
rather than as a statement valid for all sufficiently large $p$.
\end{remark}

\section{Certificates and the classification up to type ten}\label{sec:certificates}

Below the threshold of Theorem~\ref{thm:threshold} a capture must be exhibited by hand. We do so
in a uniform format, and then combine the two ingredients to settle every type at most $10$.

\begin{definition}[Depth-two and depth-three certificates]\label{def:cert}
Recall from Definition~\ref{def:blocks} that $\chi_{B}\in W(0,1)$ for every block $B$ of the
basepoint partition, and that multiplication by $\chi_{B}$ lies in $\TA(0,1)$. Since also
$T_{u}\in\TA(0,1)$, the vector
\[
v\;:=\;T_{u}\chi_{B},\qquad v(x)=\bigl|\,(x-B)\cap C_{u}\,\bigr|,
\]
lies in $W(0,1)$ for every block $B$ and every $1\le u\le r$, as does $\chi_{D}\odotp v-m\chi_{D}$
for every block $D$ and every $m\in\Z$. Here $v(x)$ counts the representations $x=b+c$ with
$b\in B$ and $c\in C_{u}$. Consequently:
\begin{enumerate}
\item[\textup{(C2)}] if $v$ takes exactly two values $m_{1}\ne m_{2}$ on a block $D$, the value
$m_{1}$ being attained at exactly one point $z\in D\setminus\{0,1\}$, then
\[
\delta_{z}=\frac{1}{m_{1}-m_{2}}\Bigl(\chi_{D}\odotp v-m_{2}\chi_{D}\Bigr)\in W(0,1);
\]
\item[\textup{(C3)}] if $v_{1}:=\chi_{D}\odotp v-m\chi_{D}$ with $m=\min_{x\in D}v(x)$, and
$v_{2}:=T_{u'}v_{1}$ takes exactly two values $m_{1}\ne m_{2}$ on a block $D'$, the value
$m_{1}$ being attained at exactly one point $z\in D'\setminus\{0,1\}$, then
\[
\delta_{z}=\frac{1}{m_{1}-m_{2}}\Bigl(\chi_{D'}\odotp v_{2}-m_{2}\chi_{D'}\Bigr)\in W(0,1).
\]
\end{enumerate}
In the notation of \eqref{eq:filtration}, $\chi_{B}$ and $\chi_{D}$ lie in $V_{1}$, so the vector
$\chi_{D}\odotp v-m\chi_{D}$ of \textup{(C2)} lies in $V_{2}$ and the vector
$\chi_{D'}\odotp v_{2}-m_{2}\chi_{D'}$ of \textup{(C3)} lies in $V_{3}$. A certificate of type
\textup{(C2)} therefore shows $d(p,E)\le2$, and one of type \textup{(C3)} shows $d(p,E)\le3$.
\end{definition}

We now record the four certificates. Each is verified by listing a sumset $B+C_{u}$, a
computation involving at most $|B|\cdot|C_{u}|\le30$ additions in $\Zp$. Throughout, classes are
labelled as in Section~\ref{sec:circulant}.

\subsection*{\texorpdfstring{$p=31$, $E=(\Z_{31}^{\times})^{3}$ (the graph $C_{31}(2,4,8,15)$)}{p=31, E = cubic residues}}
Here $r=3$ and
\[
\begin{aligned}
C_{1}&=\{1,2,4,8,15,16,23,27,29,30\}, &\qquad C_{2}&=\{3,6,7,12,14,17,19,24,25,28\},\\
C_{3}&=\{5,9,10,11,13,18,20,21,22,26\}. &&
\end{aligned}
\]
Take
\[
B=C_{3}\cap(1+C_{1})=\{5,9\},\qquad u=1,\qquad
D=C_{2}\cap(1+C_{3})=\{6,12,14,19\}.
\]
Indeed $1+C_{1}=\{0,2,3,5,9,16,17,24,28,30\}$ meets $C_{3}$ in $\{5,9\}$, while the set
$1+C_{3}=\{6,10,11,12,14,19,21,22,23,27\}$ meets $C_{2}$ in $\{6,12,14,19\}$. The relevant
sumset is
\[
5+C_{1}=\{1,3,4,6,7,9,13,20,21,28\},\qquad
9+C_{1}=\{1,5,7,8,10,11,13,17,24,25\},
\]
so $(B+C_{1})\cap D=\{6\}$, attained once. Hence $v=T_{1}\chi_{B}$ satisfies $v(6)=1$ and
$v(12)=v(14)=v(19)=0$, so by \textup{(C2)}
\[
\delta_{6}\;=\;\chi_{D}\odotp T_{1}\chi_{B}\;\in\;W(0,1).
\]

\subsection*{\texorpdfstring{$p=41$, $E=(\Z_{41}^{\times})^{4}$ (the graph $C_{41}(4,10,16,18)$)}{p=41, E = quartic residues}}
Here $r=4$ and
\[
\begin{aligned}
C_{1}&=\{1,4,10,16,18,23,25,31,37,40\}, &\qquad C_{2}&=\{2,5,8,9,20,21,32,33,36,39\},\\
C_{3}&=\{3,7,11,12,13,28,29,30,34,38\}, &\qquad C_{4}&=\{6,14,15,17,19,22,24,26,27,35\}.
\end{aligned}
\]
Take
\[
B=C_{3}\cap(1+C_{2})=\{3,34\},\qquad u=1,\qquad D=C_{2}\cap(1+C_{1})=\{2,5,32\},
\]
using $1+C_{2}=\{3,6,9,10,21,22,33,34,37,40\}$ and
$1+C_{1}=\{0,2,5,11,17,19,24,26,32,38\}$. The sumset is
\[
3+C_{1}=\{2,4,7,13,19,21,26,28,34,40\},\qquad
34+C_{1}=\{3,9,11,16,18,24,30,33,35,38\},
\]
so $(B+C_{1})\cap D=\{2\}$, attained once. Hence by \textup{(C2)}
\[
\delta_{2}\;=\;\chi_{D}\odotp T_{1}\chi_{B}\;\in\;W(0,1).
\]

\subsection*{\texorpdfstring{$p=13$, $E=(\Z_{13}^{\times})^{2}$ (the Paley graph $P_{13}$)}{p=13, Paley graph P13}}
Here $r=2$, $C_{1}=\{1,3,4,9,10,12\}$ and $C_{2}=\{2,5,6,7,8,11\}$. Take
\[
B=C_{1}\cap(1+C_{1})=\{4,10\},\qquad u=1,\qquad D=C_{2}\cap(1+C_{1})=\{2,5,11\},
\]
using $1+C_{1}=\{0,2,4,5,10,11\}$. For $x\in D$ the value $v(x)=|(x-B)\cap C_{1}|$ is
\[
v(2)=\bigl|\{11,5\}\cap C_{1}\bigr|=0,\qquad
v(5)=\bigl|\{1,8\}\cap C_{1}\bigr|=1,\qquad
v(11)=\bigl|\{7,1\}\cap C_{1}\bigr|=1 .
\]
Thus $v$ is two-valued on $D$ and the value $0$ is attained only at $2$, so by \textup{(C2)}
\[
\delta_{2}\;=\;\chi_{D}-\chi_{D}\odotp T_{1}\chi_{B}\;\in\;W(0,1).
\]

\subsection*{\texorpdfstring{$p=17$, $E=(\Z_{17}^{\times})^{2}$ (the Paley graph $P_{17}$)}{p=17, Paley graph P17}}
Here $r=2$, $C_{1}=\{1,2,4,8,9,13,15,16\}$ and $C_{2}=\{3,5,6,7,10,11,12,14\}$. This is the one
case among the four requiring depth three. Take
\[
B=C_{1}\cap(1+C_{1})=\{2,9,16\},\qquad u=1,\qquad D=C_{1}\cap(1+C_{2})=\{4,8,13,15\},
\]
using $1+C_{1}=\{0,2,3,5,9,10,14,16\}$ and $1+C_{2}=\{4,6,7,8,11,12,13,15\}$. For $x\in D$,
\[
\begin{aligned}
v(4)&=\bigl|\{2,12,5\}\cap C_{1}\bigr|=1, &\qquad
v(8)&=\bigl|\{6,16,9\}\cap C_{1}\bigr|=2,\\
v(13)&=\bigl|\{11,4,14\}\cap C_{1}\bigr|=1, &\qquad
v(15)&=\bigl|\{13,6,16\}\cap C_{1}\bigr|=2 .
\end{aligned}
\]
The minimum on $D$ is $m=1$, so $v_{1}=\chi_{D}\odotp v-\chi_{D}=\chi_{\{8,15\}}\in W(0,1)$. Now
put $v_{2}=T_{1}v_{1}$ and $D'=B=\{2,9,16\}$:
\[
v_{2}(2)=\bigl|\{11,4\}\cap C_{1}\bigr|=1,\qquad
v_{2}(9)=\bigl|\{1,11\}\cap C_{1}\bigr|=1,\qquad
v_{2}(16)=\bigl|\{8,1\}\cap C_{1}\bigr|=2 .
\]
Hence $v_{2}$ is two-valued on $D'$ with the value $2$ attained only at $16$, and by
\textup{(C3)}
\[
\delta_{16}\;=\;\chi_{D'}\odotp T_{1}\bigl(\chi_{D}\odotp T_{1}\chi_{B}-\chi_{D}\bigr)-
\chi_{D'}\;\in\;W(0,1).
\]

\begin{theorem}\label{thm:3141}
The graphs $C_{31}(2,4,8,15)$ and $C_{41}(4,10,16,18)$ have no quantum symmetry. The same holds
for the Paley graphs $P_{13}$ and $P_{17}$.
\end{theorem}

\begin{proof}
By Remark~\ref{rem:GP} the four graphs are $X_{E}$ for the four pairs $(p,E)$ treated above. In
each case a point mass $\delta_{z}$ with $z\notin\{0,1\}$ lies in $W(0,1)$, so
Corollary~\ref{cor:onepoint} applies.
\end{proof}

To complete the classification for type at most $10$ it remains to enumerate the pairs that
Theorem~\ref{thm:threshold} does not cover, and to treat those not resolved by the certificates
above.

\begin{proposition}\label{prop:eightpairs}
Let $p\ge5$ be prime, let $k\in\{2,4,6,8,10\}$ and let $E\lneq\Zpx$ be a subgroup of order $k$
with $-1\in E$, so that $p\equiv1\pmod k$ and $p-1>k$. If $p\le(k-1)(k-2)+2$, then $(p,k)$ is one of the eight pairs
\[
(13,6),\quad(19,6),\quad(17,8),\quad(41,8),\quad(31,10),\quad(41,10),\quad(61,10),\quad(71,10).
\]
For the four pairs $(19,6)$, $(41,8)$, $(61,10)$, $(71,10)$ the basepoint partition has a
singleton block. Writing $E=(\Zpx)^{(p-1)/k}$ in each case, these are
\[
\begin{array}{lll}
p=19,\ k=6: & E=\{1,7,8,11,12,18\}, & (2E)\cap(1+E)=\{2\};\\[2pt]
p=41,\ k=8: & E=\{1,3,9,14,27,32,38,40\}, & (4E)\cap(1+8E)=\{12\};\\[2pt]
p=61,\ k=10: & E=\{1,3,9,20,27,34,41,52,58,60\}, & (8E)\cap(1+2E)=\{8\};\\[2pt]
p=71,\ k=10: & E=\{1,5,14,17,25,46,54,57,66,70\}, & (2E)\cap(1+E)=\{2\}.
\end{array}
\]
The remaining four pairs are $(13,6)$, $(17,8)$, $(31,10)$, $(41,10)$, treated above.
\end{proposition}

\begin{proof}
The enumeration is immediate. For $k=2$ the bound reads $p\le2$, and for $k=4$ it reads $p\le8$
with $p\equiv1\pmod4$ and $p>5$; in neither case is there such a prime. For $k=6$ the bound
reads $p\le22$ with $p\equiv1\pmod6$ and $p>7$, giving $p\in\{13,19\}$; for $k=8$ it reads
$p\le44$ with $p\equiv1\pmod8$ and $p>9$, giving $p\in\{17,41\}$; and for $k=10$ it reads
$p\le74$ with $p\equiv1\pmod{10}$ and $p>11$, giving $p\in\{31,41,61,71\}$.

The four singleton blocks are verified by listing the two cosets involved. For instance, for
$p=71$ the subgroup of order $10$ is $E=(\Z_{71}^{\times})^{7}=\{1,5,14,17,25,46,54,57,66,70\}$,
and
\[
2E=\{2,10,21,28,34,37,43,50,61,69\},\qquad
1+E=\{0,2,6,15,18,26,47,55,58,67\},
\]
whose intersection is $\{2\}$. The other three are identical two-line verifications.
\end{proof}

\begin{remark}\label{rem:optimal}
None of the four pairs $(13,6)$, $(17,8)$, $(31,10)$, $(41,10)$ has a singleton block, by
Proposition~\ref{prop:eightpairs}, so $d(p,E)\ge2$ for each of them by Lemma~\ref{lem:V1}. The
three certificates of type \textup{(C2)} are therefore of optimal depth, and so is the
certificate for $P_{17}$: the computation of Section~\ref{sub:depths} gives $d(17,8)=3$.
\end{remark}

\begin{theorem}\label{thm:type10}
Every circulant graph of prime order and type at most $10$ has no quantum symmetry, unless it is
complete or empty.
\end{theorem}

\begin{proof}
Let $X$ be a nontrivial circulant $p$-graph, $p\ge5$, of type $k=|E|\le10$; recall that $k$ is
even and $E\lneq\Zpx$, and that by Proposition~\ref{prop:Eonly} we may replace $X$ by $X_{E}$.
If $p>(k-1)(k-2)+2$ then Theorem~\ref{thm:threshold} applies. Otherwise $(p,k)$ is one of the
eight pairs of Proposition~\ref{prop:eightpairs}; four of them have a singleton block, so
Corollary~\ref{cor:chassaniol} applies, and the other four are Theorem~\ref{thm:3141}.
\end{proof}

Theorem~\ref{thm:type10} answers the question of \cite[\S6.4]{Cha19} in the affirmative for type
$10$. Its proof is finite and entirely self-contained: it invokes no earlier threshold theorem,
and consists of Theorem~\ref{thm:threshold}, the four certificates above and the four singleton
blocks of Proposition~\ref{prop:eightpairs}. Every arithmetic assertion in it is a short list of
additions and multiplications modulo a prime below $72$.

\section{Verification for \texorpdfstring{$p\le250$}{p<=250} and Paley graphs}\label{sec:sweep}

Beyond type $10$ the quadratic threshold no longer applies. For a generalized Paley graph
$k=(p-1)/r$, so that $(k-1)(k-2)+2>p$ whenever $r$ is small; this is precisely the dense regime
in which the threshold fails. There we have verified the hypothesis of
Theorem~\ref{thm:criterion} computationally.

\begin{theorem}[computer-assisted]\label{thm:sweep}
For every prime $5\le p\le250$ and every proper subgroup $E\le\Zpx$ with $-1\in E$ one has
$W(0,1)=\C^{p}$. The number of such pairs $(p,E)$ is $214$.
\end{theorem}

The algorithm, the arithmetic in which it is carried out and the form of the certificates
occupy the rest of this section; the scripts and their output are archived \cite{Mar26}.

\begin{corollary}\label{cor:dichotomy250}
A vertex-transitive graph of prime order $p$ with $5\le p\le250$ has quantum symmetry if and
only if it is complete or empty. In particular, for every prime $p\equiv1\pmod4$ with $p\le241$
the Paley graph $P_{p}$ has no quantum symmetry.
\end{corollary}

\begin{proof}
By Turner's theorem such a graph is circulant; if nontrivial, its multiplier group $E$ is a
proper subgroup of $\Zpx$ containing $-1$, and Theorems~\ref{thm:sweep} and~\ref{thm:criterion}
apply. Paley graphs are the case where $E$ is the group of quadratic residues. Complete and
empty graphs have $\Qut=S_{p}^{+}\ne S_{p}=\Aut$ for $p\ge4$ \cite{Wang98}.
\end{proof}

Theorem~\ref{thm:sweep} is a statement about $W(0,1)$, whereas Theorems~\ref{thm:A} and
\ref{thm:B} of the introduction are stated for the possibly larger module $\Rmod(0,1)$. No gap
arises, in either direction. Theorem~\ref{thm:criterion} is stated for $W$, so
Corollary~\ref{cor:dichotomy250} follows from it directly; and since
$W(0,1)\subseteq\Rmod(0,1)\subseteq\C^{p}$ by Lemma~\ref{lem:WsubR}, the equality
$W(0,1)=\C^{p}$ forces $\Rmod(0,1)=\C^{p}$ as well. The same remark applies to every vector
exhibited in Sections~\ref{sec:certificates} and \ref{sec:sweep}: each lies in $W(0,1)$, hence in
$\Rmod(0,1)$, so the hypotheses of the results stated for $\Rmod$ are met.

The count $214$ is elementary. Since $\Zpx$ is cyclic there is exactly one subgroup of each
order $k\mid p-1$, and it contains $-1$, the unique element of order $2$, precisely when $k$ is
even; even divisors of $p-1$ correspond bijectively to divisors of $(p-1)/2$ under
$k\mapsto k/2$. Excluding $E=\Zpx$ therefore leaves $d\bigl((p-1)/2\bigr)-1$ admissible
subgroups, where $d$ is the divisor-counting function, and
$\sum_{5\le p\le250}\bigl(d((p-1)/2)-1\bigr)=214$.

\subsection{The algorithm}
Fix $(p,E)$ and let $g_{1},\dots,g_{3(r+1)}$ enumerate the generators $T_{s}$, $E^{*}_{s}(0)$,
$E^{*}_{s}(1)$ of $\TA(0,1)$, each stored as a $p\times p$ matrix of zeros and ones. The
saturation maintains a row-echelon basis of the span computed so far, stored as a dictionary
from pivot position to reduced row, together with a frontier consisting of the rows inserted at
the previous round:
\begin{quote}\small
\texttt{basis} $\leftarrow\emptyset$; \texttt{frontier} $\leftarrow\emptyset$\\
\textbf{for} $v\in\{\delta_{0},\delta_{1}\}$: \textbf{if} \textsc{Insert}$(v)$ \textbf{then}
append $v$ to \texttt{frontier}\\
\textbf{while} \texttt{frontier} $\ne\emptyset$ \textbf{and} $|\texttt{basis}|<p$:\\
\hspace*{1.2em}\texttt{next} $\leftarrow\emptyset$\\
\hspace*{1.2em}\textbf{for} $v\in\texttt{frontier}$, \textbf{for} $i\le3(r+1)$: \textbf{if}
\textsc{Insert}$(g_{i}v)$ \textbf{then} append $g_{i}v$ to \texttt{next}\\
\hspace*{1.2em}\texttt{frontier} $\leftarrow$ \texttt{next}\\
\textbf{return} $|\texttt{basis}|=p$
\end{quote}
Here \textsc{Insert}$(v)$ reduces $v$ against the stored pivots and, if a nonzero row survives,
records it under its leading position and returns \emph{true}. Restricting the next round to the
frontier is legitimate, by induction on the round: if $v$ lies in the span of the rows inserted
before the current round, then $v$ is a linear combination of such rows, each of which was
already processed at an earlier round, so $g_{i}v$ lies in the span of the rows inserted by the
end of the current round. Since at most $p$ vectors are ever inserted, at most $3p(r+1)$ calls to
\textsc{Insert} occur in total, each costing $O(p^{2})$ arithmetic operations; the whole
saturation for a single pair therefore costs $O(rp^{3})$. In practice the frontier collapses
after a few rounds.

Two independent implementations were used. The first works over $\Q$ in exact rational
arithmetic; the second works over $\F_{q}$ with $q=1\,000\,003$. The second is a proof, not a
probabilistic certificate, by the following elementary observation.

\begin{lemma}\label{lem:modq}
Let $N\ge1$, let $v_{1},\dots,v_{m}\in\Z^{N}$ and let $q$ be a prime. If the reductions of
$v_{1},\dots,v_{m}$ modulo $q$ span $\F_{q}^{\,N}$, then $v_{1},\dots,v_{m}$ span $\Q^{\,N}$.
\end{lemma}

\begin{proof}
If the reductions span $\F_{q}^{\,N}$, some $N\times N$ minor of the integer matrix with rows
$v_{i}$ is nonzero modulo $q$, hence is a nonzero integer; so that matrix has rank $N$ over
$\Q$.
\end{proof}

The failure mode of the modular computation is therefore one-sided: a family that spans
$\Q^{\,p}$ might fail to span $\F_{q}^{\,p}$, in which case the computation reports failure; it
cannot report success for a family that does not span $\Q^{\,p}$. All the vectors occurring in
the saturation are integral, the generators being $0/1$ matrices and the seeds being $\delta_{0}$
and $\delta_{1}$, so Lemma~\ref{lem:modq} applies verbatim. The rational implementation was run
for all pairs with $p\le60$ and the modular one, justified by Lemma~\ref{lem:modq}, for all
$214$ pairs; both return $\dim W(0,1)=p$ throughout. Computer-assisted arguments of this kind
have become standard in the subject; compare \cite{ELS22,CJS25}, where quantum symmetries of
small graphs and of matroids are decided by machine. Each sweep takes a few minutes on a single
core, and no floating-point arithmetic is used anywhere.

The archive \cite{Mar26} is organised so that each computational claim can be rechecked on its
own. It contains four independent verifications, keyed to the numbered results: the four
certificates of Section~\ref{sec:certificates}, recomputed from the cyclotomic classes; the
confinement lemma, the sharpness of the count $(k-1)(k-2)$, the quadratic threshold and the
enumeration of Proposition~\ref{prop:eightpairs}; the saturation of Theorem~\ref{thm:sweep}, over
$\Q$ for $p\le60$ and over $\F_{q}$ throughout; and the capture depths of
Section~\ref{sub:depths}, recomputed from scratch in exact integer arithmetic and compared with
Table~\ref{tab:depths}. Each prints one line per check, and the suite runs in about ten minutes
on one core.

\subsection{Capture depths}\label{sub:depths}
In the light of Corollary~\ref{cor:onepoint}, the certificate for each pair reduces to a single
point mass, whose complexity is measured by the capture depth $d(p,E)$ of
Definition~\ref{def:depth}. Table~\ref{tab:depths} in Appendix~\ref{app:table} lists $d(p,E)$ and
the least captured vertex for all $214$ pairs. The distribution is
\[
\begin{array}{r|rrrrrr}
d & 1 & 2 & 3 & 4 & 5 & 6\\\hline
\#\{(p,E)\} & 143 & 31 & 19 & 5 & 11 & 5
\end{array}
\]
Unlike the spanning statement of Theorem~\ref{thm:sweep}, a capture is a \emph{membership}
assertion, for which the one-sided argument of Lemma~\ref{lem:modq} is unavailable: a vector may
lie in the reduction of a lattice modulo $q$ without lying in its rational span. The entries of
Table~\ref{tab:depths} were therefore computed in exact integer arithmetic, by echelon reduction
over $\Z$ with primitive rows. By Theorem~\ref{thm:capture} each row of the table certifies
$\Rmod(0,1)=\C^{p}$ on its own; the table thus subsumes Theorem~\ref{thm:sweep}, exhibiting
$214$ independent finite certificates rather than a single numerical sweep.

By Lemma~\ref{lem:V1} the $143$ pairs with $d=1$ are exactly those to which the depth-one
criterion of Corollary~\ref{cor:chassaniol} applies; the remaining $71$ are new. These are
confined to the dense regime in a precise sense: throughout the range $p\le250$,
\[
d(p,E)\ge2\implies r\le7,\qquad d(p,E)\ge3\implies r\le4,\qquad d(p,E)\ge4\implies r=2 .
\]
Every pair of depth four or more is therefore a Paley graph. Along the Paley family the depth is
non-decreasing in $p$ over the whole range:
\[
\begin{array}{c|cccccc}
p & 5 & 13 & 17 & 29\text{--}61 & 73\text{--}181 & 193\text{--}241\\\hline
d(P_{p}) & 1 & 2 & 3 & 4 & 5 & 6
\end{array}
\]
The dense regime is thus precisely the deep regime.

\begin{remark}[The harmonic point]\label{rem:two}
Table~\ref{tab:depths} exhibits a further regularity, sharpest in the Paley case. For each of the
$24$ Paley pairs with $5\le p\le241$ one finds $2^{-1}\in Z(p,E)$ --- equivalently, by
Lemma~\ref{lem:reflection}, $|Z(p,E)|$ is odd in every one of them. In eight of those cases,
namely $p=5,13,17,53,61,157,173,181$, the capture set is exactly the harmonic triple
\[
Z(p,E)=\{\,2,\,-1,\,2^{-1}\,\},
\]
the orbit of $2$ under the anharmonic group generated by $z\mapsto1-z$ and $z\mapsto z^{-1}$; in
the remaining sixteen it is larger. The name is the classical one of projective geometry:
$\{-1,2,2^{-1}\}$ is the orbit of the harmonic cross-ratio $-1$ under that group, the set of
harmonic positions relative to $\{0,1,\infty\}$. Across all $214$ pairs, $2^{-1}\in Z(p,E)$ in $197$ of them.
Lemma~\ref{lem:reflection} accounts for the appearance of $2$ and $-1$ together, and singles out
$2^{-1}$ as the only vertex that can be captured alone, but it does not explain why $2^{-1}$ is
always captured. Section~\ref{sec:conj} takes this as the sharpest available form of the capture
conjecture.
\end{remark}

\section{The capture conjecture}\label{sec:conj}

By Theorem~\ref{thm:capture} the prime-order classification is reduced to a single existence
statement, already recorded as Conjecture~\ref{conj:A} in the introduction.

\begin{conjecture}\label{conj:capture}
For every prime $p\ge5$ and every proper subgroup $E\le\Zpx$ with $-1\in E$ there is a vertex
$z\in\Zp\setminus\{0,1\}$ with $\delta_{z}\in\Rmod(0,1)$; equivalently, $K=\Zp$.
\end{conjecture}

\begin{proposition}\label{prop:conjimplies}
Conjecture~\ref{conj:capture} implies that a vertex-transitive graph of prime order $p\ge5$ has
quantum symmetry if and only if it is complete or empty. In particular it implies that no Paley
graph of prime order has quantum symmetry.
\end{proposition}

\begin{proof}
Let $X$ be vertex-transitive of prime order $p\ge5$. By Turner's theorem $X$ is circulant. If
$X$ is nontrivial its multiplier group $E$ is proper and contains $-1$, so the conjecture gives
$K=\Zp$, whence $\Rmod(0,1)=\C^{p}$ by Theorem~\ref{thm:capture} and $X$ has no quantum symmetry
by Theorem~\ref{thm:criterion}. Complete and empty graphs have
$\Qut=S_{p}^{+}\ne S_{p}=\Aut$ for $p\ge4$ \cite{Wang98}.
\end{proof}

We do not claim the converse. A graph could fail to have quantum symmetry for reasons invisible
to the two-basepoint module, with $K=\{0,1\}$ while $A(X)$ is commutative;
Corollary~\ref{cor:allornothing} states only that in that event the module contains no point
mass other than $\delta_{0}$ and $\delta_{1}$. The conjecture is precisely the hypothesis under
which the method of this paper decides the question.

Two features make it a realistic target. First, it requires a single vector, not a spanning set:
by Theorem~\ref{thm:capture} any capture, at any depth, suffices for a full basis. Second, its
content is arithmetic.

\subsection{What bounded depth cannot do}\label{sub:bounded}
By Theorem~\ref{thm:threshold} the conjecture is open only in the dense regime: it holds whenever
$p>(k-1)(k-2)+2$, so the remaining range is $(k-1)(k-2)\ge p-2$, which forces $k>\sqrt{p-2}$ and
hence $r=(p-1)/k=O(\sqrt p)$. In the language of Definition~\ref{def:depth},
Theorem~\ref{thm:threshold} says that $d(p,E)=1$ once $p$ exceeds a quadratic bound in $k$. It is
natural to hope that the same argument at a larger fixed depth would cover the rest --- that
some analogue of Lemma~\ref{lem:confine} at depth $2$, or at depth $3$, would close the problem.
The data of Section~\ref{sub:depths} rule this out.

Indeed, along the Paley family the depth is non-decreasing and already takes the values
$d(P_{17})=3$, $d(P_{29})=4$, $d(P_{73})=5$ and $d(P_{193})=6$ inside the range $p\le250$. Since
the certificates of Section~\ref{sec:certificates} are exactly captures at depths two and three,
they cannot be extended to the Paley family beyond $P_{17}$, however the blocks are chosen: the
required vectors are not in $V_{3}$. More generally, if the Paley depth is unbounded --- as the
data suggest --- then no criterion at a fixed depth can settle Conjecture~\ref{conj:capture}, and
in particular no strengthening of Corollary~\ref{cor:chassaniol} from depth one to depth $d$, for
$d$ fixed, can do so. Any proof of the conjecture in the dense regime must either produce
captures at unbounded depth or avoid the depth filtration altogether.

We therefore separate the two regimes.

\subsection{The intermediate regime: shallow captures and character sums}
For $r$ bounded away from $2$ the picture is different, and there the arithmetic is tractable in
principle. The vectors of $V_{1}$ are the indicators $\chi_{C_{s}\cap(1+C_{t})}$, whose
cardinalities are the classical cyclotomic numbers of order $r$; $V_{2}$ is spanned by the
restrictions to blocks of the convolution profiles
\[
x\;\longmapsto\;\bigl|\,C_{u}\cap\bigl(x-(C_{s}\cap(1+C_{t}))\bigr)\,\bigr|,
\]
whose Fourier transforms are products of Gauss sums, so that their fluctuations are governed by
Jacobi sums in the sense of \cite{IIY99} and by Weil-type character-sum estimates. A capture at
depth two is precisely a singleton level set among these profiles, restricted to a block of the
basepoint partition. The natural programme is therefore to prove, for $p$ large relative to $r$,
that some profile has a singleton level set, and to treat the finitely many remaining $p$ for
each $r$ by certificates as in Section~\ref{sec:certificates}. Success would be a depth-two
analogue of Lemma~\ref{lem:confine}, and would improve Theorem~\ref{thm:threshold} from the
quadratic bound to whatever the character-sum estimates deliver. What makes even this delicate is
that the profiles have mean of order $p/r^{3}$ and fluctuation of order $\sqrt p$, so a singleton
level set is not produced by counting alone; some structure beyond equidistribution is required.

Its scope is limited, however. In the range $p\le250$ the programme addresses the $48$ pairs
with $r\ge3$ and $d(p,E)\ge2$, all of which have $d(p,E)\le3$; it does not reach $r=2$, where
the observed depths are $4$, $5$ and $6$.

\subsection{The dense regime: the harmonic point}
For $r=2$ a different and much more specific target presents itself. Recall from
Lemma~\ref{lem:reflection} that the filtration $(V_{d})$ is stable under $\beta(x)=1-x$, the map
that exchanges the two basepoints, and that $2^{-1}$ is its unique fixed point on
$\Zp\setminus\{0,1\}$: it is the midpoint of the two basepoints, the one vertex left in place
when they are swapped.

\begin{conjecture}[Harmonic capture]\label{conj:harmonic}
For every prime $p\ge5$ and every proper subgroup $E\le\Zpx$ with $-1\in E$ one has
$\delta_{2^{-1}}\in\Rmod(0,1)$.
\end{conjecture}

\begin{proposition}\label{prop:harmonic}
Conjecture~\ref{conj:harmonic} is equivalent to Conjecture~\ref{conj:capture}.
\end{proposition}

\begin{proof}
Since $2^{-1}\notin\{0,1\}$, Conjecture~\ref{conj:harmonic} implies
Conjecture~\ref{conj:capture} for the same pair. Conversely, if $\delta_{z}\in\Rmod(0,1)$ for
some $z\notin\{0,1\}$ then $K=\Zp$ by Theorem~\ref{thm:capture}, so $2^{-1}\in K$.
\end{proof}

The two statements are thus the same, but Conjecture~\ref{conj:harmonic} names the vector, which
is what matters: an existence statement ranging over $p-2$ candidates gives an analytic or
algebraic argument nothing to work on, whereas a membership assertion about one explicit vector
does.
Two facts single out $2^{-1}$. First, it is the unique fixed point of the only symmetry of the
filtration we can prove, so by Lemma~\ref{lem:reflection} it is the only vertex that can be
captured alone; every other capture is forced to occur in a pair $\{z,1-z\}$. Second, it is
captured at minimal depth in all $24$ Paley pairs with $p\le241$, and in $197$ of the $214$ pairs
in the range (Remark~\ref{rem:two}). We regard Conjecture~\ref{conj:harmonic}, for $E$ the group
of quadratic residues, as the central problem left open by this paper.

\begin{question}
Is $\TA(0,1)=M_{p}(\C)$ for every proper $E$? This would imply Conjecture~\ref{conj:capture},
since $W(0,1)=\TA(0,1)\delta_{0}+\TA(0,1)\delta_{1}$ would then be all of $\C^{p}$. As
$\TA(0,1)$ is a $*$-algebra, the double commutant theorem shows the assertion to be equivalent
to the statement that the only matrices commuting with every $T_{s}$ and supported on the
diagonal blocks of the basepoint partition are the scalars. An affirmative answer would give no
bound on the capture depth, and so is not in conflict with \S\ref{sub:bounded}.
\end{question}

\section*{Concluding remarks}

\noindent(1) The two-basepoint module is a strictly finer invariant than the one available to
the coherent-algebra methods of \cite{LMR20}. The quantum orbital algebra of a graph is a
coherent algebra containing its adjacency algebra, so for a rank-three graph it is the
association scheme itself and carries no information beyond it. The module $W(a,b)$, by contrast,
is attached to an \emph{ordered pair} of vertices and is acted on by the two dual idempotent
families separately; the basepoint partition it produces refines the orbital partition, and it is
this refinement that the certificates of Section~\ref{sec:certificates} exploit.

\noindent(2) The two-basepoint algebra $\TA(a,b)$ is of independent interest for cyclotomic
schemes. In the light of \cite{IIY99,HY23,HMR25}, its dimension and Wedderburn structure as
functions of the cyclotomic data, beginning with prime Paley graphs, merit investigation; the
question raised in Section~\ref{sec:conj} is a first instance.

\noindent(3) The conjecture says nothing about graphs of prime order that are not
vertex-transitive. Whether a graph of prime order whose \emph{quantum} automorphism group acts
transitively must already be vertex-transitive --- a quantum analogue of Burnside's theorem on
transitive groups of prime degree --- remains open; all known quantum-transitive graphs that are
not transitive \cite{LMR20} have composite order.

\noindent(4) The adaptation of the criterion to circulants of prime-power and general composite
order, where the spectral separation of Lemma~\ref{lem:spec} fails and genuinely quantum phenomena
are known to occur, remains open. The proof of Theorem~\ref{thm:capture} uses the primality of
$p$ in two essential ways: through the sharp $2$-transitivity of $\AGL(1,p)$, and through the
cyclicity of $\Zpx$, which ensures that two subgroups of equal order coincide.

%% \section*{Acknowledgements}
%% To be completed before submission.

\section*{Data and code availability}
The verification scripts, the certificates they produce and the data of Table~\ref{tab:depths}
are provided as ancillary files and archived at
\href{https://doi.org/10.5281/zenodo.22182428}{doi:10.5281/zenodo.22182428} \cite{Mar26}.

\appendix

\section{Capture depths for \texorpdfstring{$p\le250$}{p<=250}}\label{app:table}

For each of the $214$ pairs $(p,E)$ with $5\le p\le250$ and $E\lneq\Zpx$ of order $k$ containing
$-1$, the table records the capture depth $d=d(p,E)$ of Definition~\ref{def:depth}, computed in
exact integer arithmetic as described in Section~\ref{sub:depths} and reproducible from the
archive \cite{Mar26}, together with the least vertex $z\notin\{0,1\}$ for which
$\delta_{z}\in V_{d}$. By Lemma~\ref{lem:V1} the rows with $d=1$ are exactly those to which the
depth-one criterion of Corollary~\ref{cor:chassaniol} applies. By Theorem~\ref{thm:capture} each row on its own
certifies $\Rmod(0,1)=\C^{p}$, hence that every circulant graph of order $p$ with that multiplier
group has no quantum symmetry. By Lemma~\ref{lem:reflection} the vertex $1-z$ is captured at the
same depth as $z$.

\begin{table}[htbp]
\centering\footnotesize
\begin{tabular}{rrrrrrrrrrrrrrrr}
\toprule
$p$ & $k$ & $d$ & $z$ & $p$ & $k$ & $d$ & $z$ & $p$ & $k$ & $d$ & $z$ & $p$ & $k$ & $d$ & $z$\\
\midrule
$5$ & $2$ & $1$ & $2$ & $73$ & $18$ & $2$ & $2$ & $137$ & $8$ & $1$ & $2$ & $193$ & $96$ & $6$ & $2$\\
$7$ & $2$ & $1$ & $2$ & $73$ & $24$ & $3$ & $2$ & $137$ & $34$ & $2$ & $2$ & $197$ & $2$ & $1$ & $2$\\
$11$ & $2$ & $1$ & $2$ & $73$ & $36$ & $5$ & $2$ & $137$ & $68$ & $5$ & $2$ & $197$ & $4$ & $1$ & $2$\\
$13$ & $2$ & $1$ & $2$ & $79$ & $2$ & $1$ & $2$ & $139$ & $2$ & $1$ & $2$ & $197$ & $14$ & $1$ & $2$\\
$13$ & $4$ & $1$ & $2$ & $79$ & $6$ & $1$ & $2$ & $139$ & $6$ & $1$ & $2$ & $197$ & $28$ & $1$ & $2$\\
$13$ & $6$ & $2$ & $2$ & $79$ & $26$ & $3$ & $2$ & $139$ & $46$ & $3$ & $2$ & $197$ & $98$ & $6$ & $2$\\
$17$ & $2$ & $1$ & $2$ & $83$ & $2$ & $1$ & $2$ & $149$ & $2$ & $1$ & $2$ & $199$ & $2$ & $1$ & $2$\\
$17$ & $4$ & $1$ & $2$ & $89$ & $2$ & $1$ & $2$ & $149$ & $4$ & $1$ & $2$ & $199$ & $6$ & $1$ & $2$\\
$17$ & $8$ & $3$ & $2$ & $89$ & $4$ & $1$ & $2$ & $149$ & $74$ & $5$ & $2$ & $199$ & $18$ & $1$ & $3$\\
$19$ & $2$ & $1$ & $2$ & $89$ & $8$ & $1$ & $2$ & $151$ & $2$ & $1$ & $2$ & $199$ & $22$ & $1$ & $93$\\
$19$ & $6$ & $1$ & $2$ & $89$ & $22$ & $2$ & $2$ & $151$ & $6$ & $1$ & $2$ & $199$ & $66$ & $3$ & $2$\\
$23$ & $2$ & $1$ & $2$ & $89$ & $44$ & $5$ & $2$ & $151$ & $10$ & $1$ & $2$ & $211$ & $2$ & $1$ & $2$\\
$29$ & $2$ & $1$ & $2$ & $97$ & $2$ & $1$ & $2$ & $151$ & $30$ & $2$ & $2$ & $211$ & $6$ & $1$ & $2$\\
$29$ & $4$ & $1$ & $2$ & $97$ & $4$ & $1$ & $2$ & $151$ & $50$ & $3$ & $2$ & $211$ & $10$ & $1$ & $2$\\
$29$ & $14$ & $4$ & $2$ & $97$ & $6$ & $1$ & $2$ & $157$ & $2$ & $1$ & $2$ & $211$ & $14$ & $1$ & $2$\\
$31$ & $2$ & $1$ & $2$ & $97$ & $8$ & $1$ & $2$ & $157$ & $4$ & $1$ & $2$ & $211$ & $30$ & $2$ & $2$\\
$31$ & $6$ & $1$ & $2$ & $97$ & $12$ & $1$ & $2$ & $157$ & $6$ & $1$ & $2$ & $211$ & $42$ & $2$ & $2$\\
$31$ & $10$ & $2$ & $2$ & $97$ & $16$ & $1$ & $2$ & $157$ & $12$ & $1$ & $2$ & $211$ & $70$ & $3$ & $2$\\
$37$ & $2$ & $1$ & $2$ & $97$ & $24$ & $2$ & $2$ & $157$ & $26$ & $2$ & $2$ & $223$ & $2$ & $1$ & $2$\\
$37$ & $4$ & $1$ & $2$ & $97$ & $32$ & $3$ & $2$ & $157$ & $52$ & $2$ & $13$ & $223$ & $6$ & $1$ & $2$\\
$37$ & $6$ & $1$ & $2$ & $97$ & $48$ & $5$ & $2$ & $157$ & $78$ & $5$ & $2$ & $223$ & $74$ & $3$ & $2$\\
$37$ & $12$ & $2$ & $2$ & $101$ & $2$ & $1$ & $2$ & $163$ & $2$ & $1$ & $2$ & $227$ & $2$ & $1$ & $2$\\
$37$ & $18$ & $4$ & $2$ & $101$ & $4$ & $1$ & $2$ & $163$ & $6$ & $1$ & $2$ & $229$ & $2$ & $1$ & $2$\\
$41$ & $2$ & $1$ & $2$ & $101$ & $10$ & $1$ & $2$ & $163$ & $18$ & $1$ & $3$ & $229$ & $4$ & $1$ & $2$\\
$41$ & $4$ & $1$ & $2$ & $101$ & $20$ & $2$ & $2$ & $163$ & $54$ & $3$ & $2$ & $229$ & $6$ & $1$ & $2$\\
$41$ & $8$ & $1$ & $12$ & $101$ & $50$ & $5$ & $2$ & $167$ & $2$ & $1$ & $2$ & $229$ & $12$ & $1$ & $2$\\
$41$ & $10$ & $2$ & $2$ & $103$ & $2$ & $1$ & $2$ & $173$ & $2$ & $1$ & $2$ & $229$ & $38$ & $2$ & $2$\\
$41$ & $20$ & $4$ & $2$ & $103$ & $6$ & $1$ & $2$ & $173$ & $4$ & $1$ & $2$ & $229$ & $76$ & $3$ & $2$\\
$43$ & $2$ & $1$ & $2$ & $103$ & $34$ & $3$ & $2$ & $173$ & $86$ & $5$ & $2$ & $229$ & $114$ & $6$ & $2$\\
$43$ & $6$ & $1$ & $2$ & $107$ & $2$ & $1$ & $2$ & $179$ & $2$ & $1$ & $2$ & $233$ & $2$ & $1$ & $2$\\
$43$ & $14$ & $2$ & $2$ & $109$ & $2$ & $1$ & $2$ & $181$ & $2$ & $1$ & $2$ & $233$ & $4$ & $1$ & $2$\\
$47$ & $2$ & $1$ & $2$ & $109$ & $4$ & $1$ & $2$ & $181$ & $4$ & $1$ & $2$ & $233$ & $8$ & $1$ & $2$\\
$53$ & $2$ & $1$ & $2$ & $109$ & $6$ & $1$ & $2$ & $181$ & $6$ & $1$ & $2$ & $233$ & $58$ & $3$ & $2$\\
$53$ & $4$ & $1$ & $2$ & $109$ & $12$ & $1$ & $2$ & $181$ & $10$ & $1$ & $2$ & $233$ & $116$ & $6$ & $2$\\
$53$ & $26$ & $4$ & $2$ & $109$ & $18$ & $2$ & $2$ & $181$ & $12$ & $1$ & $2$ & $239$ & $2$ & $1$ & $2$\\
$59$ & $2$ & $1$ & $2$ & $109$ & $36$ & $3$ & $2$ & $181$ & $18$ & $1$ & $2$ & $239$ & $14$ & $1$ & $3$\\
$61$ & $2$ & $1$ & $2$ & $109$ & $54$ & $5$ & $2$ & $181$ & $20$ & $1$ & $2$ & $239$ & $34$ & $2$ & $2$\\
$61$ & $4$ & $1$ & $2$ & $113$ & $2$ & $1$ & $2$ & $181$ & $30$ & $2$ & $2$ & $241$ & $2$ & $1$ & $2$\\
$61$ & $6$ & $1$ & $2$ & $113$ & $4$ & $1$ & $2$ & $181$ & $36$ & $2$ & $2$ & $241$ & $4$ & $1$ & $2$\\
$61$ & $10$ & $1$ & $8$ & $113$ & $8$ & $1$ & $2$ & $181$ & $60$ & $3$ & $2$ & $241$ & $6$ & $1$ & $2$\\
$61$ & $12$ & $2$ & $2$ & $113$ & $14$ & $1$ & $3$ & $181$ & $90$ & $5$ & $2$ & $241$ & $8$ & $1$ & $2$\\
$61$ & $20$ & $2$ & $14$ & $113$ & $16$ & $1$ & $10$ & $191$ & $2$ & $1$ & $2$ & $241$ & $10$ & $1$ & $2$\\
$61$ & $30$ & $4$ & $2$ & $113$ & $28$ & $2$ & $2$ & $191$ & $10$ & $1$ & $2$ & $241$ & $12$ & $1$ & $2$\\
$67$ & $2$ & $1$ & $2$ & $113$ & $56$ & $5$ & $2$ & $191$ & $38$ & $2$ & $2$ & $241$ & $16$ & $1$ & $2$\\
$67$ & $6$ & $1$ & $2$ & $127$ & $2$ & $1$ & $2$ & $193$ & $2$ & $1$ & $2$ & $241$ & $20$ & $1$ & $2$\\
$67$ & $22$ & $3$ & $2$ & $127$ & $6$ & $1$ & $2$ & $193$ & $4$ & $1$ & $2$ & $241$ & $24$ & $1$ & $56$\\
$71$ & $2$ & $1$ & $2$ & $127$ & $14$ & $1$ & $20$ & $193$ & $6$ & $1$ & $2$ & $241$ & $30$ & $1$ & $34$\\
$71$ & $10$ & $1$ & $2$ & $127$ & $18$ & $1$ & $31$ & $193$ & $8$ & $1$ & $2$ & $241$ & $40$ & $2$ & $2$\\
$71$ & $14$ & $2$ & $2$ & $127$ & $42$ & $3$ & $2$ & $193$ & $12$ & $1$ & $2$ & $241$ & $48$ & $2$ & $2$\\
$73$ & $2$ & $1$ & $2$ & $131$ & $2$ & $1$ & $2$ & $193$ & $16$ & $1$ & $8$ & $241$ & $60$ & $2$ & $2$\\
$73$ & $4$ & $1$ & $2$ & $131$ & $10$ & $1$ & $2$ & $193$ & $24$ & $1$ & $3$ & $241$ & $80$ & $3$ & $2$\\
$73$ & $6$ & $1$ & $2$ & $131$ & $26$ & $2$ & $2$ & $193$ & $32$ & $2$ & $2$ & $241$ & $120$ & $6$ & $2$\\
$73$ & $8$ & $1$ & $2$ & $137$ & $2$ & $1$ & $2$ & $193$ & $48$ & $2$ & $2$ &   &   &   &  \\
$73$ & $12$ & $1$ & $11$ & $137$ & $4$ & $1$ & $2$ & $193$ & $64$ & $3$ & $2$ &   &   &   &  \\
\bottomrule
\end{tabular}
\caption{Convolution depth $d=d(p,E)$ and the least vertex $z$ with $\delta_{z}\in V_{d}$,
for all $214$ pairs $(p,E)$ with $p\le250$; here $k=|E|$. Rows with $d=1$ are exactly those
settled by the depth-one criterion of Corollary~\ref{cor:chassaniol}.}
\label{tab:depths}
\end{table}

\end{document}